\documentclass[a4paper,10pt]{article}
\usepackage{amsmath}
\usepackage{amsthm}
\usepackage{amssymb}
\usepackage{amsfonts}
\usepackage{graphics,graphicx}
\usepackage{color}
\usepackage{url}
\usepackage{enumerate}
\usepackage{xspace}
\usepackage{xspace}
\usepackage[noend]{algpseudocode}
\usepackage{float}
\usepackage[caption = false]{subfig}
\usepackage{soul}
\usepackage{textcomp}

\usepackage{fancyvrb}
\usepackage[normalem]{ulem}

\let\counterwithin\Relax

\usepackage{chngcntr}
\usepackage{xcolor}
\usepackage{tikz}
\usepackage{pgf,tikz,pgfplots}
\usetikzlibrary{arrows}

\DeclareSymbolFont{AMSb}{U}{msb}{m}{n}
\DeclareMathSymbol{\N}{\mathbin}{AMSb}{"4E}
\DeclareMathSymbol{\Z}{\mathbin}{AMSb}{"5A}
\DeclareMathSymbol{\R}{\mathbin}{AMSb}{"52}
\DeclareMathSymbol{\Q}{\mathbin}{AMSb}{"51}
\DeclareMathSymbol{\C}{\mathbin}{AMSb}{"43}

\newcommand{\ee}{\textrm{e}}
\newcommand{\ii}{\textrm{i}}

\newcommand{\Ai}{\mathrm{Ai}}

\newcommand{\pone}{P1}
\newcommand{\ptwo}{P2}

\newtheorem{lemma}{Lemma}[section]
\newtheorem{theorem}[lemma]{Theorem}
\newtheorem{proposition}[lemma]{Proposition}
\newtheorem{corollary}[lemma]{Corollary}

\counterwithin{figure}{section}
\counterwithin{table}{section}

\newcommand{\Pstat}{\mathcal{P}_{\rm stat}} %
\newcommand{\Pexit}{\mathcal{P}_{\rm exit}} %
\newcommand{\Pentr}{\mathcal{P}_{\rm entr}} %
\newcommand{\Pendp}{\mathcal{P}_{\rm endp}} %

\newcommand{\cV}{\mathcal{V}}

\newcommand{\dd}{\ensuremath{\operatorname{d}\!}}	
\newcommand*{\e}{\mathrm{e}}

\newcommand\dist{\mathop{\rm dist}}

\definecolor{dhcol}{rgb}{0.8,0,0}

\definecolor{agcol}{rgb}{0,0,0.8}

\title{
Numerical evaluation of oscillatory integrals via automated steepest descent contour deformation} 
\author{%
A.\ Gibbs\footnotemark[1], %
D.\ P.\ Hewett\footnotemark[1], 
D.\ Huybrechs\footnotemark[2], 
 \\[6pt]
\footnotemark[1] Department of Mathematics, University College London, London, UK\\
\footnotemark[2] Department of Computer Science, KU Leuven, Leuven, Belgium
}
\date{}

\begin{document}
	
	\maketitle
	\begin{abstract}
		Steepest descent methods combining complex contour deformation with numerical quadrature provide an efficient and accurate 		%
		approach for the evaluation of highly oscillatory integrals. 
However, unless the phase function governing the oscillation is particularly simple, their application requires a significant amount of a priori analysis and expert user input, to determine the appropriate contour deformation, and to deal with the non-uniformity in the accuracy of standard quadrature techniques associated with the coalescence of stationary points (saddle points) with each other, or with the endpoints of the original integration contour. 
In this paper we present a novel algorithm for the numerical evaluation of oscillatory integrals with general polynomial phase functions, which automates the contour deformation process and avoids the difficulties typically encountered with coalescing stationary points and endpoints. 
The inputs to the algorithm are simply the phase and amplitude functions, the endpoints and orientation of the original integration contour, and a small number of numerical parameters. 
By a series of numerical experiments we demonstrate that the algorithm is accurate and efficient over a large range of frequencies, even for examples with a large number of coalescing stationary points and with endpoints at infinity. As a particular application, we use our algorithm to evaluate cuspoid canonical integrals from scattering theory. A Matlab implementation of the algorithm is made available and is called PathFinder.
	\end{abstract}

\section{Introduction}\label{sec:intro}
	
In this paper we consider numerical evaluation of the integral
\begin{equation}\label{eq:I}
I = \int_\Gamma f(z)\e^{\ii\omega g(z)}\dd{z},
\end{equation}
where $\Gamma$ is a contour in $\C$, possibly starting and/or ending at infinity, $f$ and $g$ are functions of a complex variable, and $\omega>0$ is a frequency parameter. 
Such integrals arise in numerous application areas, particularly in wave phenomena and quantum mechanics, and are generally challenging to evaluate numerically, especially when $\omega$ is large, because the presence of the exponential factor $\e^{\ii\omega g(z)}$ means that the integrand may undergo rapid oscillations and/or significant variations in amplitude along the integration contour. 

{The numerical evaluation of oscillatory integrals is a well-studied topic, and a number of different approaches have been developed for integrals of the form \eqref{eq:I}, applicable under certain assumptions on $f$, $g$ and $\Gamma$. For an overview of the field and links to relevant literature we refer the reader to \cite{DeHuIs:18} and \cite[\S36.15]{DLMF}.}

When $f$ and $g$ are analytic, Cauchy's theorem provides the possibility of deforming the integration contour so as to make numerical evaluation easier. This is the basis of \emph{steepest descent (SD) methods}, in which one aims to deform $\Gamma$ onto a contour, or, more typically, a union of contours, which we term the \emph{steepest descent (SD) deformation}, on which $\Re[g(z)]$ is constant, so that the exponential factor $\e^{\ii\omega g(z)}$ is no longer oscillatory. By the Cauchy-Riemann equations, these contours coincide with the steepest descent curves of $-\Im[g(z)]$, and they connect endpoints of the original integration contour, valleys at infinity (sectors in which the integrand decays rapidly as $|z|\to\infty$), and \textit{stationary points }of $g$, which are points $\xi\in\C$ at which $g'(\xi)=0$. 
\footnote{Stationary points are often referred to as ``saddle points'' because they are saddle points of the functions $\Im[g(z)]$ and $\Re[g(z)]$, which cannot possess local maxima or minima by the maximum modulus principle.}  Along each SD contour, away from stationary points the integrand typically decays exponentially, with the rate of decay increasing with increasing $\omega$, 
and as $\omega\to\infty$ the value of the integral is dominated by local contributions close to the endpoints of $\Gamma$ and any stationary points traversed by the SD deformation. 
In the \emph{asymptotic} steepest descent method (described e.g.\ in \cite{bleistein1986asymptotic,wong2001asymptotic}), one exploits this to obtain an asymptotic expansion for the integral, valid as $\omega\to\infty$, by performing a local Taylor expansion of the integrand around the endpoints and relevant stationary points, and reducing the local integrals along the SD contours to simpler integrals that can be expressed in terms of known special functions. 

In the \emph{numerical} steepest descent (NSD) method (described e.g.\ in \cite[\S5]{DeHuIs:18}) one evaluates the integrals along the SD contours numerically. This involves numerically tracing an appropriate segment of each SD contour in the SD deformation %
and applying suitable numerical quadrature rules to evaluate the associated contributions to the integral. 

In principle, NSD is a highly accurate and efficient method for evaluating integrals of the form \eqref{eq:I} for moderate or large $\omega$. 
Indeed, under appropriate assumptions, the NSD method outputs approximations which, for a fixed number of quadrature points $N$, are asymptotic to \eqref{eq:I} as $\omega\to\infty$, 
with the asymptotic accuracy improving with increasing $N$ (see, e.g., \cite[Thm~5.7]{DeHuIs:18}).
Furthermore, if $f$ and $g$ are sufficiently well behaved it can also be the case that the NSD approximation converges to \eqref{eq:I} as $N\to\infty$, for fixed $\omega>0$, with a cost that remains bounded 
as $\omega\to\infty$. %

In practice, however, applying the NSD method to an integral of the form \eqref{eq:I} often requires significant expert user input.
This is because: %
\begin{enumerate}
\item[(\pone)] Determining the SD contour deformation corresponding to a given $g$ and $\Gamma$ requires careful a priori analysis. 
\item[(\ptwo)] 
Parametrizing SD contours from or near stationary points, and evaluating integrals along them, is fraught with numerical difficulties, 
especially 
when stationary points are close to other stationary points or endpoints of $\Gamma$.
\end{enumerate}
The issues described in (\pone) and (\ptwo) are particularly troublesome when one wishes to evaluate \eqref{eq:I} for multiple instances of a phase function $g(z)=g(z,\mathbf{c})$ depending on a set of parameters $\mathbf{c}\in \C^q$. This is because, firstly, the number and location of the stationary points, and the nature of the SD deformation, have to be determined for each different value of $\mathbf{c}$, and, secondly, stationary points may coalesce as $\mathbf{c}$ approaches certain regions in parameter space, leading to a non-uniformity in the accuracy of the resulting NSD approximations. 

The problem of stationary point coalescence in the context of NSD was studied in detail in \cite{HuJuLe:19} in the special case of the cubic phase function $g(z,c)=\frac{z^3}{3}-cz$, for $c\in \C$, which for $c\neq 0$ has a pair of order one stationary points which coalesce as $c\to0$ (at $z=0$) into a single stationary point of order two for $c=0$.\footnote{The \textit{order} of a stationary point $\xi$ is the multiplicity of $\xi$ as a root of $g'$.} 
In this case, the SD deformation and contour parametrization were carried out manually by analytically inverting the phase (illustrating (P1)), but the resulting integrals were found to be nearly singular for small $c$, leading to poor accuracy of standard NSD approximations (illustrating (P2)). 
It was shown in \cite{HuJuLe:19} how to construct a family of non-standard quadrature rules for this integral which perform uniformly well for $c\approx 0$ using complex-valued Gaussian quadrature, producing quadrature nodes that in general lie off the SD deformation. In principle, similar rules could be developed for more complicated coalescences involving higher order stationary points and/or endpoints of $\Gamma$. However, for each type of coalescence a bespoke quadrature rule would have to be developed, and a general catalogue of such rules is not yet available in the literature.

In contrast to \cite{HuJuLe:19}, our aim is not to develop an optimized method for a specific instance of \eqref{eq:I}, but rather to present a relatively simple algorithm that can evaluate \eqref{eq:I} accurately, for a general class of $f$ and $g$, without the need for expert user input or a priori analysis, even in the case of coalescing stationary points, thus addressing problems (P1) and (P2). 
Our specific focus in this paper is on the case where $f$ is entire and $g$ is a polynomial. The extension of our approach to more general cases where $f$ and/or $g$ have pole or branch point singularities 
is the subject of ongoing research.\footnote{{We note that NSD-based methods for cases where $f$ has a singularity at or close to the endpoint of the integration contour have been presented previously - see e.g.\ \cite{huybrechs2007sparse,gibbs2020fast}, where such methods are applied in the context of integral equation methods for high frequency scattering.}} 
Necessarily, in aiming for generality and robustness we will sacrifice some efficiency. 
{In particular, in contrast to standard NSD, the error in the approximations produced by our algorithm does not in general decay as $\omega\to\infty$.} 
Nonetheless, our {algorithm} is designed to be rapidly convergent as $N\to\infty$ with approximately $\omega$-independent {error} and {$\omega$-independent computational cost}, and the fact that this is realised in practice is demonstrated by extensive numerical experiments in \S\ref{sec:Numerics}. %

Our algorithm follows the basic principles of NSD, combining complex contour deformation with numerical quadrature. However, in contrast to standard NSD our algorithm does not trace SD contours directly from stationary points. Instead, stationary points are enclosed in a bounded ``non-oscillatory region'' within which the integrand is guaranteed to undergo at most a fixed number of oscillations. The original contour $\Gamma$ is replaced by a ``quasi-SD deformation'' comprising a union of straight-line contours in the non-oscillatory region, for which numerical quadrature is straightforward, and SD contours outside the non-oscillatory region, on which standard NSD quadrature techniques can be applied. 
By excluding a neighbourhood of the stationary points from the region in which SD contours are traced, our algorithm avoids the problems mentioned in (\ptwo) associated with stationary-point/stationary-point and/or stationary-point/endpoint coalescence. This not only ``uniformizes'' the accuracy of our algorithm compared to standard NSD, but it also enables us to tackle the problem (\pone) by automating the contour deformation step. For the latter, we first perform low accuracy SD contour tracing outside the non-oscillatory region to build a graph describing the global connections (via SD contours) between the endpoints of $\Gamma$, the different components of the non-oscillatory region, and the valleys at infinity, and then determine the quasi-SD deformation using a shortest path algorithm, before refining the accuracy of the SD contour tracing at the quadrature stage.  

One other problem with standard NSD is that it typically degenerates as $\omega\to 0$, because the quadrature points diverge to infinity \cite[\S5.2.4]{DeHuIs:18}. This issue has been addressed in the special case $g(z)=z$ for bounded $\Gamma$ in \cite{asheim2013gaussian,celsus2021kissing}; however, it remains an open problem for general $g(z)$. 
Our algorithm is well-behaved in the limit as $\omega\to 0$ for general polynomial $g(z)$, since it reduces to standard non-oscillatory quadrature for sufficiently small $\omega$ for any bounded $\Gamma$. 

Our algorithm is implemented in the open-source Matlab code ``PathFinder'', available at \textsf{github.com/AndrewGibbs/PathFinder} \cite{PathFinder}. The basic user input to the code is a function handle for the amplitude \texttt{f}, the coefficients of the polynomial phase \texttt{g}, endpoints \texttt{a} and \texttt{b} (complex numbers, or angles in the case of infinite endpoints), the frequency parameter \texttt{omega}, and a parameter \texttt{N} controlling the number of quadrature points to be used. Approximating the integral \eqref{eq:I} using PathFinder can be done with the following Matlab command:
\begin{align}
\label{eq:PathFinderCode}
\texttt{PathFinder(a,b,f,g,omega,N,\textquotesingle infcontour\textquotesingle,[A B])}
\end{align}
Here \texttt{\textquotesingle infcontour\textquotesingle} is an optional input for which the user should supply a Boolean array \texttt{[A B]} (whose default value is \texttt{[false false]}) such that \texttt{A} (respectively \texttt{B}) is  \texttt{true} if the endpoint \texttt{a} (resp.\ \texttt{b}) is infinite and \texttt{false} if it is finite. Examples of PathFinder code will be given in \S\ref{sec:Numerics}.  
Advanced users can also adjust a small number of other tuning parameters, whose role will be discussed during the presentation of our algorithm. 

An outline of the paper is as follows. 
In \S\ref{sec:Algorithm} we provide a detailed description of our algorithm, first presenting an overview of the main steps, and then providing details of how each step is realised in PathFinder. 
In \S\ref{sec:Theory} we present some theoretical results underpinning our approach. 
In \S\ref{sec:Details} we discuss some further implementation details, and  
in \S\ref{sec:Numerics} we exhibit numerical results demonstrating the performance of our algorithm on a range of challenging integrals with large $\omega$ and complicated stationary point configurations.

We end this introduction by remarking that integrals with coalescing stationary points are of fundamental importance in numerous applications, including the study of optics and high frequency (short wavelength) acoustics, where they describe the wave field in the vicinity of geometrical singularities (or ``catastrophes'') in the classical ray-theoretic framework, Kelvin's celebrated ship-wave problem, and the theory of molecular collisions in quantum mechanics and theoretical chemistry. 
A catalogue of such integrals, along with links to relevant literature, can be found in 
\cite[\S36]{DLMF}. 
In \S\ref{sec:Applications} we show how PathFinder can be applied to accurately calculate these types of integrals. 

\section{Algorithm description}
\label{sec:Algorithm}
In this section we present our algorithm for the numerical approximation of \eqref{eq:I} when $f$ is entire\footnote{When $\Gamma$ is infinite we additionally implicitly assume that 
$f$ is sufficiently well-behaved at infinity {(i.e., does not grow too fast at infinity in the relevant directions)} for the integral \eqref{eq:I} to converge. Note that in many cases of interest, numerical evaluation of the integral to high accuracy after path deformation only requires $f$ to be analytic in a small (and shrinking with increasing $\omega$) neighbourhood of the stationary points.} 
and 
$g$ is a polynomial.

We start with some definitions and basic facts. 
Let 
	\begin{equation}\label{eq:gfull}
	g(z) = \sum_{j=0}^J\alpha_jz^j,
	\end{equation}
for some $J\in \N$, $J\geq 1$, and $\alpha_j\in\C$, $j=0,\ldots,J$, with $\alpha_J\neq 0$.
Then $g$ has at most $J-1$ \textit{stationary points}, which are the solutions of
\begin{align}
 \label{eq:StatPointEqn}
  g'(z) = \sum_{j=1}^J j\alpha_jz^{j-1}=0. 
\end{align}
We denote the set of all stationary points by $\Pstat$. 
We define the \textit{valleys} at infinity to be the sectors of angular width $\pi/J$ centred on the angles
\begin{equation}\label{eq:val}
		\cV := {\left\{\frac{(2(m-1)+1/2)\pi - \arg{(\alpha_J)}}{J}:\quad m=1,\ldots,J \right\}}.
		\end{equation}
These have the property that if $z=r\e^{\ii \theta}$ with $\theta\in (v-\pi/(2J),v+\pi/(2J))$ for some $v\in \cV$ then $\e^{\ii \omega g(z)}\to 0$ as $r\to\infty$.  
For each $\eta\in\C\setminus\Pstat$ there exists a unique SD contour $\gamma_{\eta}$ beginning at $\eta$ and ending either at a stationary point $\xi\in\Pstat$ or at a valley $v\in \cV$, on which $\Re g(z)= \Re g(\eta)$ for $z\in \gamma_\eta$ (see, e.g.,\ \cite{Bo:52a}). 

We let $\Pendp$ denote the set of finite endpoints of the integration contour $\Gamma$, which could have zero, one or two elements. We assume for now that any infinite endpoint of $\Gamma$ is at one of the valleys $v\in\cV$; see \S\ref{sec:InfiniteEndpoints} for extensions.  
We now provide a high-level overview of our algorithm. 
The following steps will be explained in more detail in sections \ref{sec:statpoints}-\ref{sec:evaluation}.

\begin{enumerate}
	\item\label{step:getSPs} Compute the set of stationary points $\Pstat$ (the solutions of \eqref{eq:StatPointEqn}). 
	\item\label{step:defineradii}
For each $\xi\in\Pstat$, select a radius $r_\xi>0$ for which the function $\ee^{\ii \omega g(z)}$ is considered %
``non-oscillatory'' on the closed ball $\Omega_\xi$ of radius $r_\xi$ centred at $\xi$. %
These balls may overlap. However, if two balls overlap significantly, indicating near coalescence, one of the stationary points (along with its associated ball) is removed from the set $\Pstat$. 
This removal process continues recursively until no pair of balls is judged to overlap too much. 
	We call $\{\Omega_\xi\}_{\xi\in \Pstat}$ the \emph{non-oscillatory balls}, and their union 
\begin{align}
	\label{eq:NonOscDef}
\Omega:=\bigcup_{\xi\in \Pstat}\Omega_\xi
	\end{align}
	 the \emph{non-oscillatory region}. 
	\item\label{step:steepestexits} 
	Find the local minima of $|\ee^{\ii \omega g(z)}|$ 
	on the boundary of the non-oscillatory region $\Omega$. 
We call these points \emph{exits}, and 
denote by $\Pexit$ the set of all exits.

\item\label{step:tracepaths} For each %
$\eta\in\Pexit\cup (\Pendp \setminus \Omega)$, 
trace the SD contour $\gamma_\eta$ from $\eta$, and determine whether 
\begin{itemize}
\item[(i)] $\gamma_\eta$ enters $\Omega$ at some point $z\in \partial\Omega\setminus\{\eta\}$, or 
\item[(ii)] $\gamma_\eta$ converges towards a valley $v\in \cV$ without entering $\Omega$. 
\end{itemize}
We call points $z\in\partial\Omega$ determined in case (i) \textit{entrances}, and denote by $\Pentr$ the set of all entrances. 
	\item\label{step:graph} Construct a graph $G$ with a vertex for each of the elements of $\Pstat$, $\Pendp$, $\Pexit$, $\Pentr$ and $\cV$. 
Add edges between the vertices of $G$ as follows:
\begin{itemize}
\item For each $\xi\in \Pstat$, add an edge between each pair of elements of $(\Pstat\cup \Pendp\cup\Pexit \cup \Pentr)\cap \Omega_\xi$.
\item For each pair $\xi,\xi'\in \Pstat$, $\xi\neq \xi'$, for which $\Omega_\xi\cap \Omega_{\xi'}\neq\emptyset$, add an edge between $\xi$ and $\xi'$, if not already added in the previous step. 
\item For each $\eta\in \Pexit\cup (\Pendp \setminus \Omega)$, add an edge between $\eta$ and the entrance $z\in \Pentr$ or the valley $v\in \cV$ to which the SD contour $\gamma_\eta$ leads.
\end{itemize}
Find the shortest path (in the graph-theoretic sense) between the vertices corresponding to the endpoints of $\Gamma$. 

\item\label{step:quadrature} Generate quadrature nodes and weights for the evaluation of each of the contour integrals corresponding to the edges in the shortest path. For an edge between two points in the non-oscillatory region, use a straight-line contour. For an edge between an exit or an endpoint of $\Gamma$ to an entrance or a valley, use a refined version of the SD contour traced in step 4. The union of all the contours corresponding to the edges of the shortest path defines the ``quasi-SD deformation'' of the original integration contour. 
Finally, use the quadrature nodes and weights to approximate the integrals over the contours in the quasi-SD deformation 
and sum them to obtain an approximation of the original integral \eqref{eq:I}.
\end{enumerate}

In Figures \ref{fig:intro_eg} and \ref{fig:intro_eg_graph} we illustrate the outcome of the above steps 
for the particular choice of phase function 
\begin{equation}\label{eq:intro_eg_phase}
	\begin{split}
g(z)=&\frac{z^7}{7}+z^6\,\left(\frac{7}{20}+\frac{13}{30}{}\mathrm{i}\right)+z^5\,\left(-\frac{1047}{2000}+\frac{543}{1000}{}\mathrm{i}\right)+z^4\,\left(-\frac{4409}{8000}-\frac{5077}{8000}{}\mathrm{i}\right)\\&+z^3\,\left(\frac{711}{2000}-\frac{4441}{6000}{}\mathrm{i}\right)+z^2\,\left(\frac{237}{800}-\frac{207}{800}{}\mathrm{i}\right)+z\,\left(\frac{63}{1000}-\frac{77}{2000}{}\mathrm{i}\right)
\end{split}
\end{equation}
and the parameters $\omega=40$, $a=-1.5$, $b=2$, $N=10$, using the default parameter set for PathFinder (see Table \ref{tab:Parameters}). 
For this choice of $g$ there is one order 2 stationary point and 4 order one stationary points. In Figure \ref{fig:intro_eg} we plot these stationary points, along with their non-oscillatory balls, and the SD contours traced from the exits. 
Such plots can be generated in PathFinder by adding the optional \texttt{\textquotesingle plot\textquotesingle} flag. The ball centred at the stationary point $\xi=-\ii$ contains two entrances, reached by SD contours from the balls above.  
In Figure \ref{fig:intro_eg_graph} we plot the graph $G$, 
using the optional PathFinder input flag \texttt{\textquotesingle plot graph\textquotesingle}. This graph, in addition to edges corresponding to the SD contours shown in Figure \ref{fig:intro_eg}, contains edges corresponding to contours between points in the non-oscillatory region, including connections within the two overlapping balls. 
The shortest path between $a$ and $b$, which is highlighted with thick lines in Figure \ref{fig:intro_eg_graph}, corresponds to the quasi-SD deformation, the integral over which is equal to \eqref{eq:I} by Cauchy's Theorem.
The integral is discretised using $N$ quadrature points on each contour in the quasi-SD deformation that makes a non-negligible contribution to the integral (see \S\ref{sec:evaluation}) - these points are plotted in Figure \ref{fig:intro_eg} in red.

The process of computing all the SD contours and the selection of a subset thereof via the shortest path algorithm addresses problem (P1). 
Surrounding stationary points by balls, and only tracing SD contours outside the balls, means that we avoid having to determine the local structure of the SD contours and compute integrals along them near stationary points, 
addressing problem (P2).

\begin{figure}[pt!]
	\centering
\subfloat[SD contours and quasi-SD deformation]{\includegraphics[width=0.8\linewidth]{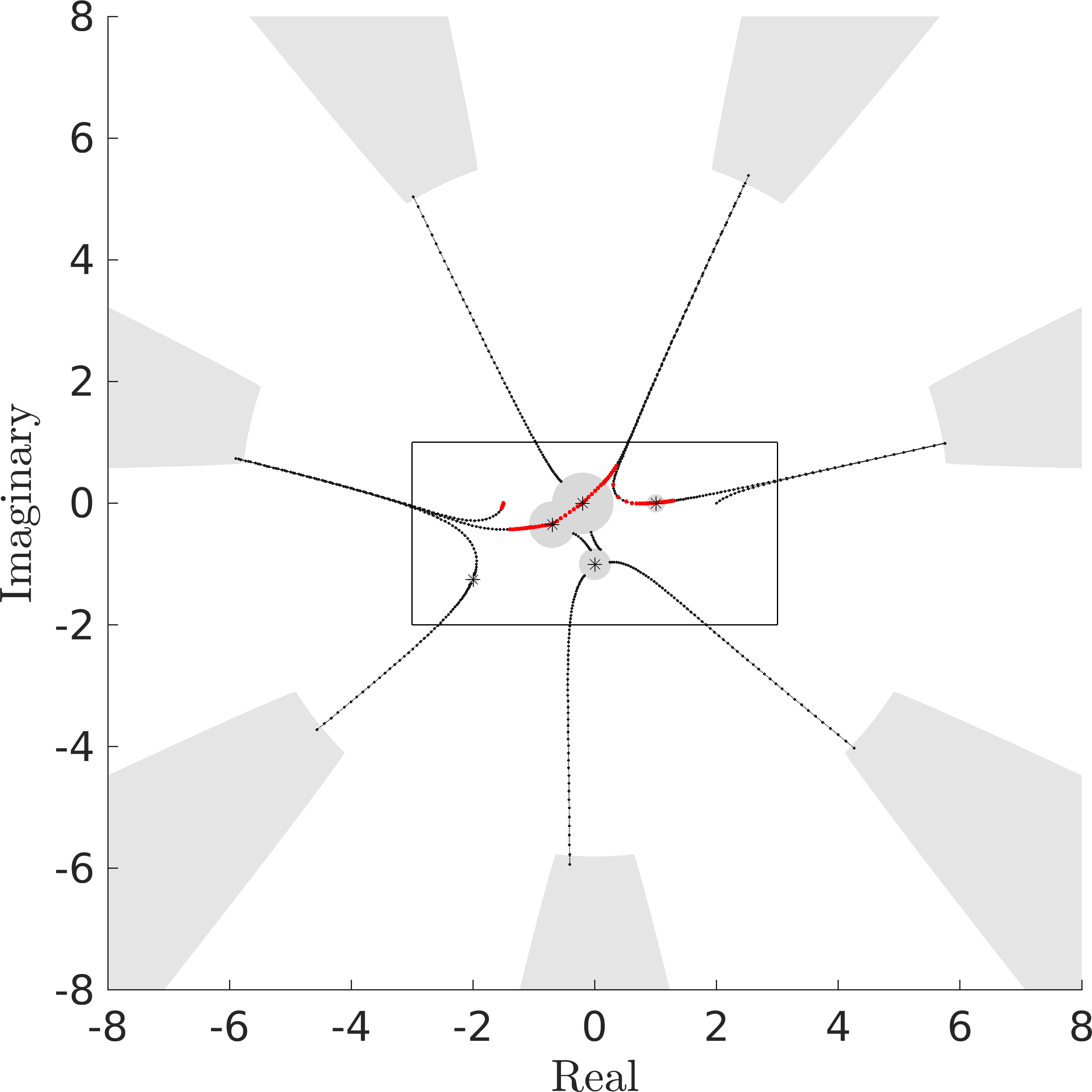}}\\
\subfloat[Zoom of boxed region near the origin]{\includegraphics[width=\linewidth]{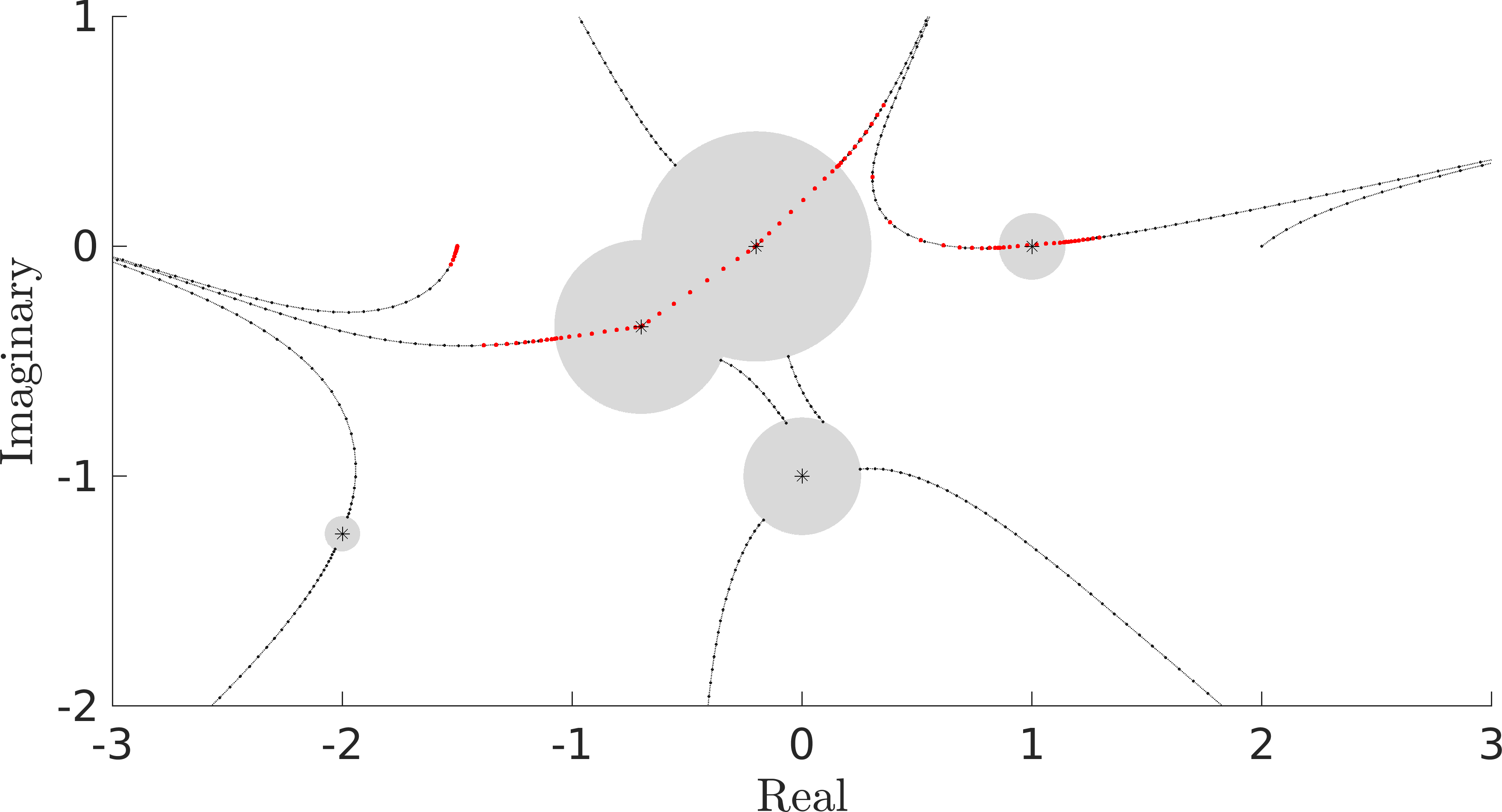}}\\
	\caption{Output of algorithm when applied with phase \eqref{eq:intro_eg_phase}, $\omega=40$, $a=-1.5$, $b=2$, $N=10$, and the default parameter set (see Table \ref{tab:Parameters}). Here we observe stationary points {$\xi\in\Pstat$} (black stars) surrounded by balls {$\Omega_\xi$} (grey) {whose union is the non-oscillatory region $\Omega$}, SD contours (black lines) traced from exits {$\eta\in\Pexit$} and finite endpoints {$\eta\in\Pendp$ to either valleys $v\in\mathcal{V}$ at infinity or to entrances $z\in \Pentr$}, and quadrature points (red points) allocated along the appropriate contours in the quasi-SD deformation. 
The ``region of no return'' {$R_v$} (see \S\ref{sec:NoReturn}) around {each of} the valleys {$v\in\mathcal{V}$} at infinity is also shaded grey. %
}
\label{fig:intro_eg}
\end{figure}

\begin{figure}[pt!]
	\centering
\subfloat[Graph $G$]{\includegraphics[width=\linewidth]{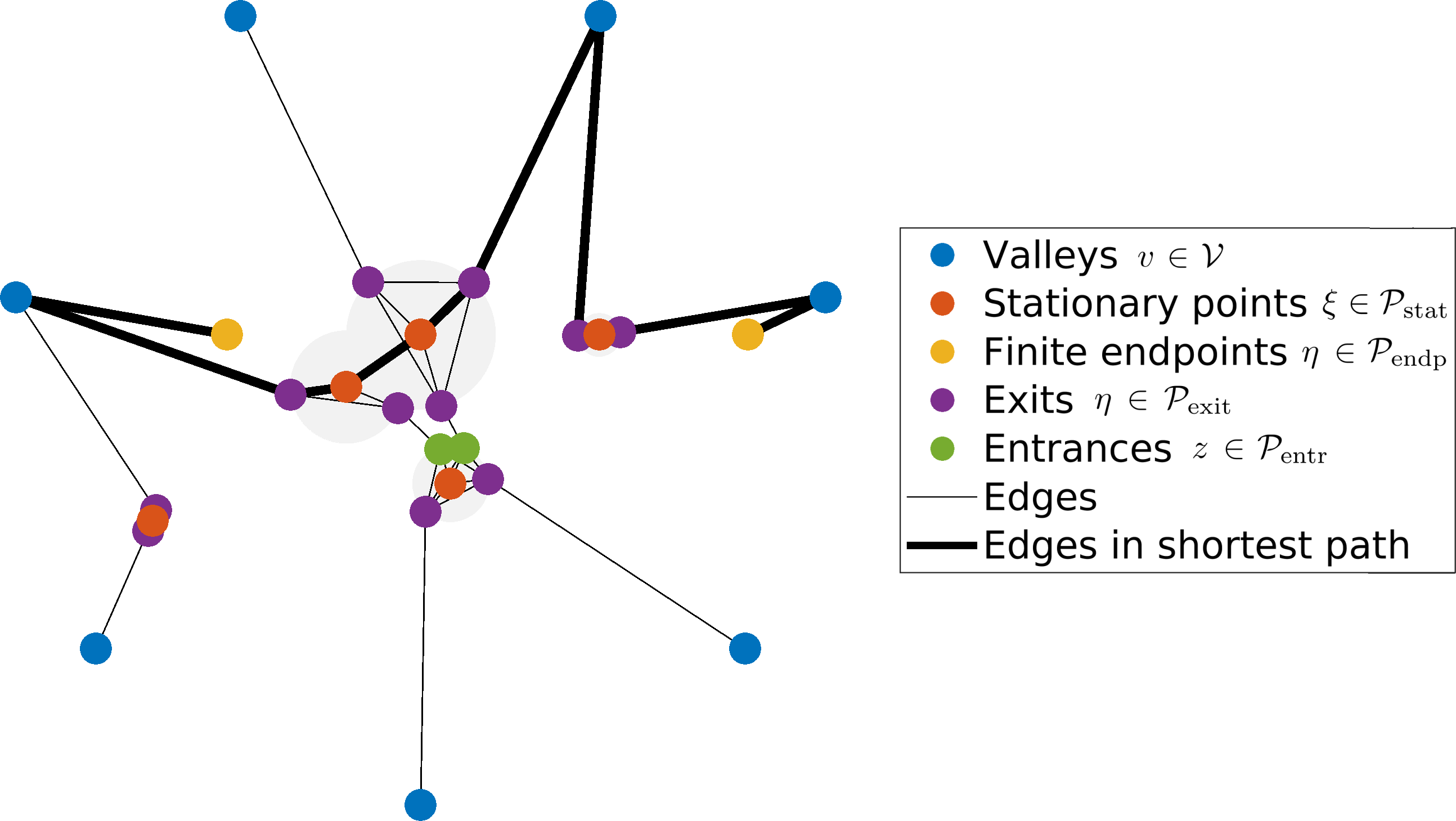}}\\
\subfloat[Zoom of the graph $G$]{\includegraphics[width=.6\linewidth]{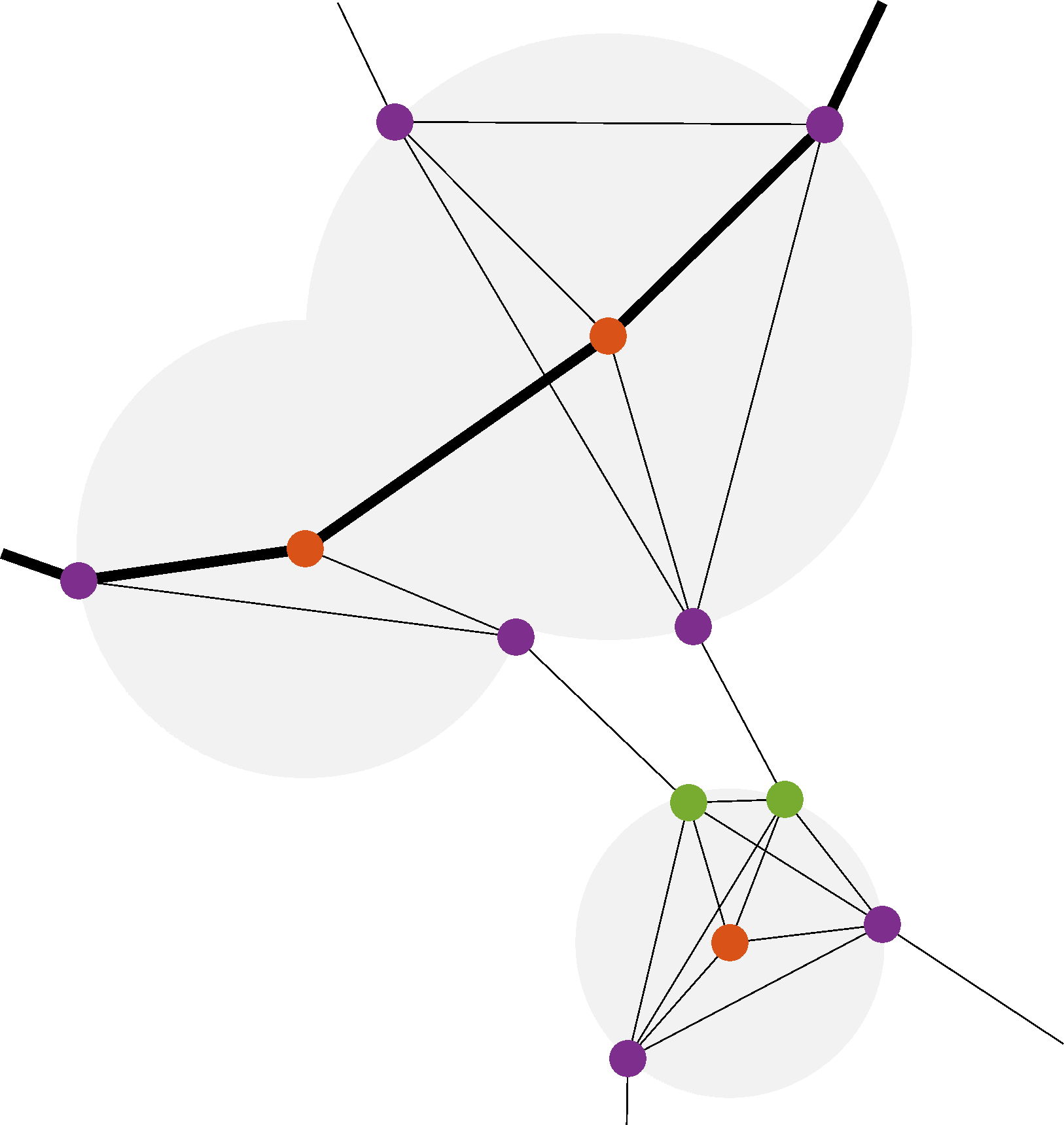}}

\caption{The graph $G$ corresponding to the problem considered in Figure \ref{fig:intro_eg}. The thick line represents the shortest path between the endpoints {of $\Gamma$}, which in this case are both finite. The balls (shaded grey) are included for ease of comparison with Figure \ref{fig:intro_eg}. The lower figure zooms in on the centre of the upper figure, showing the multiple edges that are constructed inside the balls.}\label{fig:intro_eg_graph}
\end{figure}

In the following subsections we provide further details of how we carry out the steps outlined above in PathFinder. 
\subsection{Step 1 - Computing stationary points}
\label{sec:statpoints}
Computing the stationary points of $g$ (the roots of $g'(z)$) requires us to find the complex roots of the polynomial \eqref{eq:StatPointEqn}. In our implementation we compute stationary points using the Matlab \verb|roots| command, which applies a companion matrix approach. We note that obtaining highly accurate values for the positions of stationary points is not critical to our algorithm, since the stationary points are enclosed in the non-oscillatory region and we never trace SD contours from them. Indeed, the difficulty in distinguishing numerically between multiple roots and roots of higher order contributes to the motivation for considering such non-oscillatory regions.

\subsection{Step 2 - Defining the non-oscillatory region}
\label{sec:DefNonOsc}
The non-oscillatory region $\Omega$ was defined in \eqref{eq:NonOscDef} to be a union of balls centred at the elements of $\Pstat$. We choose the radii of the balls as follows. First fix some user-defined constant $C_{\rm ball}>0$. Then, given $\xi\in \Pstat$, 
define
\begin{align}
\label{eq:rxiDef}
r_\xi:= \max\{r>0: |z-\xi|\leq r\Rightarrow \omega|g(z)-g(\xi)|\leq C_{\rm ball}\}. 
\end{align}
This definition enforces an upper bound on the number of oscillations within each ball. 
Accordingly, the region $\Omega$ shrinks to the stationary points as $\omega\to\infty$ and expands to fill the whole complex plane as $\omega\to 0$. 

In our implementation we approximate $r_\xi$ numerically as follows. Let $N_{\rm ball}\in \N$ be a user-defined parameter. For each $n\in\{1,\ldots,N_{\rm ball}\}$ we consider the ray $\{z=\xi + r\ee^{\ii 2\pi n/N_{\rm ball}}, \,r>0\}$, and compute the smallest positive root $r_n>0$ of the function $u_n(r):=\omega^2|g(\xi + r\ee^{\ii 2\pi n/N_{\rm ball}})-g(\xi)|^2-C_{\rm ball}^2$, which is a polynomial in $r$ of degree $2J$. 
For this root-finding problem we use the Matlab \verb|roots| command; in case this command produces no positive real roots (because of stability issues) we resort to a bisection approach instead.  
We then take as our approximation to $r_\xi$ the positive number $\max_{n\in\{1,\ldots,N_{\rm ball}\}} r_n$.

When elements of $\Pstat$ are close it is natural to amalgamate their respective non-oscillatory balls. To do this systematically we adopt an iterative approach. 
Let $\delta_{\rm ball}>0$ be a user-defined parameter. 
\begin{itemize}
\item 
For each pair $\xi_1,\xi_2\in \Pstat$ compute
\[d_{\xi_1,\xi_2}:=|\xi_1-\xi_2|/\max(r_{\xi_1},r_{\xi_2}).\]
\item
If $\min_{\xi_1,\xi_2} d_{\xi_1,\xi_2}<\delta_{\rm ball}$  let $\xi_1,\xi_2$ be a pair realising the minimum.
Remove from $\Pstat$ whichever of $\xi_1,\xi_2$ has the smaller associated ball radius (or choose arbitrarily between them if $r_{\xi_1}=r_{\xi_2}$).
\item Repeat the previous step 
until either $\min_{\xi_1,\xi_2} d_{\xi_1,\xi_2}\geq \delta_{\rm ball}$, or there is only one element of $\Pstat$ remaining.
\end{itemize} 

\subsection{Step 3 - Determining the exits}
The exits associated with each $\xi\in \Pstat$ are defined to be the local minima on $\partial\Omega_\xi\setminus \bigcup_{\xi'\in \Pstat,\xi'\neq \xi}\Omega_{\xi'}^\circ$ of the function $|\ee^{\ii \omega g(z)}|$, equivalently of the function $-\Im g(z)$. 

For each $\xi\in \Pstat$ the function $-\Im g(z)$ restricted to the $\partial\Omega_\xi$ is a trigonometric polynomial. Using this fact, in our implementation we determine the local minima of $-\Im g(z)$ on $\partial\Omega_\xi$ by finding the roots of the derivative of $-\Im g(z)$ in the angular direction (which is also a trigonometric polynomial) by the companion matrix approach of \cite[\S2.2]{boyd2006computing}, and keep only the real roots corresponding to local minima. 
We discard all those minima corresponding to points inside $\bigcup_{\xi'\in \Pstat,\xi'\neq \xi}\Omega_{\xi'}^\circ$, and add the remaining minima to the set $\Pexit$. 

\subsection{Step 4 - Tracing the SD contours}
\label{sec:SDTracing}
Given $\eta\in\Pexit\cup (\Pendp \setminus \Omega)$, the SD contour $\gamma_{\eta}$ beginning at $\eta$ is the unique curve on which $\Re g(z)$ is constant, with $-\Im g(z)$ decreasing along $\gamma_\eta$. It can be parametrized in terms of a parameter $p\geq 0$ as $z=h_\eta(p)$, where $h_\eta(p)$ is defined implicitly by
\begin{align}
\label{eq:SDparam}
g(h_{\eta}(p)) = g(\eta) + 
\ii p,
 \qquad h_{\eta}(0)=\eta.
\end{align}
Differentiating \eqref{eq:SDparam} with respect to $p$ gives
\begin{align}
\label{eq:ODE}
h_\eta'(p) = 
\frac{\ii}{g'(h_{\eta}(p))}
=:F(h_\eta(p)), \qquad h_{\eta}(0)=\eta,
\end{align}
which is a first order ODE initial value problem for $h_\eta(p)$. 
By solving \eqref{eq:ODE} numerically one can trace the contour $\gamma_{\eta}$ until it either (i) enters the non-oscillatory region $\Omega$, or (ii) one can be sure that it will tend to a valley $v\in \cV$, without entering $\Omega$. 
For (ii) we appeal to the theoretical result in Theorem \ref{thm:NoReturn}, which provides a ``region of no return'' $R_v$ associated with each valley $v\in\cV$ for which it is guaranteed that if an SD contour enters $R_v$ it will never leave $R_v$, and will converge to $v$. 

Staying away from stationary points ensures that the factor $1/g'$ in the right-hand side of \eqref{eq:ODE} does not get too large. 

In our implementation we trace the SD contour using a predictor-corrector approach, combining a forward Euler step for \eqref{eq:ODE} and a Newton iteration for \eqref{eq:SDparam}, to generate approximations 
$h_\eta^{(n)}\approx h_\eta(p_n)$ 
on a mesh $0=p_0<p_1<p_2<\ldots<p_{n_{\rm max}}$, where the total number of steps $n_{\rm max}$ is determined as part of the algorithm, as discussed below. 

As the initial value we take $h_\eta^{(0)}=\eta$. Then, given $h_\eta^{(n)}$, to compute $h_\eta^{(n+1)}$ we first apply a forward Euler step for the ODE \eqref{eq:ODE}, with adaptive step length
\[p_{n+1}-p_n = \delta_{\rm ODE}\min\left(2\frac{|g'(h_\eta^{(n)})|^2}{|g''(h_\eta^{(n)})|},|g'(h_\eta^{(n)})|
\dist(h_\eta^{(n)},\Pstat)\right),\]
where $\delta_{\rm ODE}\in(0,1)$ is a user-specified parameter. The first argument of the minimum is included to ensure stability of the solver - note that $F'(h) = -\frac{\ii g''(h)}{(g'(h))^2}$ and we might
expect instability if the local step length were as large as $2/|F'(h)| = 2\frac{|g'(h)|^2}{|g''(h)|}$. 
The second argument is included to ensure that the solver ``slows down'' as it approaches the non-oscillatory region $\Omega$, so that we can detect whether $\gamma_\eta$ enters $\Omega$ or not. To ensure that $|h_\eta^{(n+1)}-h_\eta^{(n)}| \leq \delta_{\rm ODE} d$, where $d:=\dist(h_\eta^{(n)},\Pstat)=\min_{\xi\in \Pstat}|h_\eta^{(n)}-\xi|$, we require that $p_{n+1}-p_n\leq \frac{\delta_{\rm ODE} d}{|F(h_\eta^{(n)})|}=\delta_{\rm ODE} d|g'(h_\eta^{(n)})|$. This also ensures that $h_\eta^{(n+1)}$ remains far enough from $\Pstat$, so that \eqref{eq:ODE} doesn't get too large.

After each forward Euler step, we correct $h_\eta^{(n+1)}$ by running a Newton iteration to enforce \eqref{eq:SDparam} (with $p=p_{n+1}$ fixed), 
until the Newton step size $|\frac{g(h_\eta^{(n+1)})-g(\eta)-\ii p_{n+1}}{g'(h_\eta^{(n+1)})}|$ is smaller than $\delta_{\rm coarse}\dist(h_\eta^{(n+1)},\Pstat)$, for some user-specified tolerance $\delta_{\rm coarse}>0$.

We repeat this process for $n=0,1,2,\ldots$ until either 
\begin{itemize}
\item[(i)] $h_\eta^{(n)}\in \Omega_\xi$ for some $\xi\in \Pstat$, in which case we add $z=h_\eta^{(n)}$ to the set $\Pentr$ of entrances. Note that in general the point $z=h_\eta^{(n)}$ will lie inside $\Omega_\xi^\circ$ rather than on $\partial\Omega_\xi$, but will be closer to $\partial\Omega_\xi$ the smaller $\delta_{\rm ODE}$ is; 
\item[] or
\item[(ii)] $h_\eta^{(n)}\in R_v$ for some $v\in \cV$, in which case, by Theorem \ref{thm:NoReturn}, $\gamma_\eta$ converges to the valley $v$. 
Here the ``region of no return'' $R_v$ is defined by
\begin{align}
\label{eq:RegOfNoReturn}
R_v:=\{z\in\C: |\arg z-v|_{2\pi}
<\pi/(2J) \text{ and }G(|z|,|\arg z-v|_{2\pi})>0\},
\end{align}
where
\begin{align}
\label{eq:AbsTheta}
|\theta|_{2\pi}:=\min_{m\in\Z}|\theta-2\pi m|,
\end{align}
and, for $r>0$ and $\theta\in(0,\pi/(2J))$, 
\begin{align}
\label{eq:RegFunction}
G(r,\theta) := J|\alpha_J|r^{J-1}\min\left(1/\sqrt{2},\cos{J\theta}\right) - 
\sum_{j=1}^{J-1} j|\alpha_j|r^{j-1}.
\end{align}
For further explanation of the meaning of $R_v$ see \S\ref{sec:NoReturn} below.

A necessary condition for a point $z$ to lie in $R_v$ is that 
$|z|\geq r_*$, where $r_*>0$ is the unique positive solution of the polynomial equation $G(r_*,\pi/(4J))=0$, i.e.\ the solution of
\[ \frac{J|\alpha_J|r_*^{J-1}}{\sqrt{2}}= 
\sum_{j=1}^{J-1} j|\alpha_j|r_*^{j-1}.\]
Having found $r_*$ once and for all (using the Matlab \texttt{roots} commnad), to check that a point $z$ lies in $R_v$ we first check 
that $|z|\geq r_*$. 
If so, we then check that $|\arg z-v|_{2\pi}<\pi/(2J)$. If so, we then check that $G(|z|,|\arg z-v|_{2\pi})>0$. 
The point of introducing $r_*$ is so that we don't compute $G(|z|,|\arg z-v|_{2\pi})$ unless absolutely necessary.
\end{itemize}

In either case, tracing of the SD contour stops and we set $n_{\rm max}=n$ for this contour.

\subsection{Step 5 - Finding the shortest path}

The construction of the graph $G$ requires no further explanation. 
To find the shortest path in $G$ between the endpoints of the original contour $\Gamma$ we apply the standard Dijkstra shortest path algorithm  \cite[\S10.6.2]{Ro:19}.

\subsection{Step 6 - Evaluating the contour integrals}
\label{sec:evaluation}

The quasi-SD contour deformation corresponding to the graph-theoretic shortest path between the endpoints of $G$ calculated during step 5 involves integrals over three types of contour: 
\begin{itemize}
\item[\quad\textbf{Type 1}:] Straight line contours between points in the non-oscillatory region;
\item[\quad\textbf{Type 2}:] Infinite SD contours from exits/endpoints to valleys;
\item[\quad\textbf{Type 3}:] Finite SD contours from exits/endpoints to entrances.
\end{itemize}

Some of these contours will make a larger contribution to the value of the original integral \eqref{eq:I} than others. It is natural to neglect contours that make a negligibly small contribution. 
In our implementation, we only compute the contribution from a contour $\gamma$ in the quasi-SD deformation if at least one of the finite endpoints $\eta$ of $\gamma$ satisfies 
$|\ee^{\ii \omega g(\eta)}|/M>\delta_{\rm quad}$, 
where $\delta_{\rm quad}\geq 0$ is a small, user-specified parameter and 
\[ M:=\max |\ee^{\ii \omega g(\xi)}|,\]
where the maximum is taken over all $\xi\in\Pstat\cup\Pendp\cup\Pexit$ appearing in the shortest path corresponding to the quasi-SD deformation. 

In our implementation, for Type 1 contours we use Gauss-Legendre quadrature, for Type 2 contours we use either Gauss-Laguerre quadrature (which is the default choice in PathFinder) or truncated Gauss-Legendre quadrature, 
and for Type 3 contours we use (possibly truncated) Gauss-Legendre quadrature, as detailed below. By default our implementation uses the same number $N$ of quadrature points on each contour in the quasi-SD deformation whose contribution we compute, regardless of the type of integral (we comment on this in \S\ref{sec:NumberOfQuadPoints}). 
Accordingly, if $N_{\rm cont}$ is the number of these contours then the total number of quadrature points used in the algorithm, $N_{\rm tot}$, is given by 
\begin{align}
\label{eq:Ntot}
N_{\rm tot} = N N_{\rm cont}.
\end{align}

\subsubsection{Evaluation of integrals over Type 1 contours}
\label{sec:Type1}
Let $z_0,z_1\in\Omega$, and let $\gamma$ be the straight-line contour in $\C$ starting at $z_0$ and ending at $z_1$, parametrized by 
\begin{align}
\label{eq:zDef}
z_{[z_0,z_1]}(t) = \frac{1}{2}\Big((z_1-z_0)t + (z_0+z_1)\Big), \qquad t\in[-1,1].
\end{align}
Given $N\in \N$, let $t_m^{\rm Leg}$ and $w_m^{\rm Leg}$, for $m=1,\ldots,N$, denote the nodes and weights for standard $N$-point Gauss-Legendre quadrature on $[-1,1]$. 
Our quadrature approximation to the integral over $\gamma$ is then:
\begin{align}
\label{eq:GauLegApplication}
\int_{\gamma}f(z)\ee^{\ii\omega g(z)}\,\dd z \approx \frac{z_1-z_0}{2} \sum_{m=1}^{N} w_m^{\rm Leg} f(z_{[z_0,z_1]}(t_m^{\rm Leg}))\ee^{\ii\omega g(z_{[z_0,z_1]}(t_m^{\rm Leg}))}. 
\end{align}

\subsubsection{Evaluation of integrals over Type 2 contours}
\label{sec:Type2}
Let $\eta\in \Pexit\cup(\Pendp\setminus\Omega)$ be such that the SD contour $\gamma$ from $\eta$ leads to a valley. Parametrizing $\gamma$ by \eqref{eq:SDparam}, for $p\in[0,\infty)$, noting \eqref{eq:ODE}, and rescaling $p=\tilde{p}/\omega$, we have
\begin{align}
\label{eq:SDIntegralInfinite}
\int_{\gamma}f(z)\ee^{\ii\omega g(z)}\,\dd z = \frac{\ee^{\ii \omega g(\eta)}}{\omega}\int_0^{\infty} \tilde{f}(\tilde{p})\ee^{-\tilde{p}}\,\dd \tilde{p}, 
\end{align}
where 
\[\tilde{f}(\tilde{p}):=f(h_\eta(\tilde{p}/\omega))h_\eta'(\tilde{p}/\omega)=\ii\frac{f(h_\eta(\tilde{p}/\omega))}{g'(h_\eta(\tilde{p}/\omega))}. \]
By tracing contours outside of $\Omega$, the contours remain a positive distance from $\Pstat$. This ensures that $\tilde{f}$ is analytic in a complex neighbourhood of $[0,\infty)$. 

By default, PathFinder evaluates the integral on the right-hand side of \eqref{eq:SDIntegralInfinite} by Gauss-Laguerre quadrature. Let $t_m^{\rm Lag}$ and $w_m^{\rm Lag}$, for $n=1,\ldots,N$, denote the standard Gauss-Laguerre nodes and weights on $[0,\infty)$. 
Our quadrature approximation to the integral over $\gamma$ is then:
\begin{align}
\label{eq:GauLagApplication}
\int_{\gamma}f(z)\ee^{\ii\omega g(z)}\,\dd z \approx \frac{\ee^{\ii \omega g(\eta)}}{\omega} \sum_{m=1}^{N}w_m^{\rm Lag}\tilde{f}(t_m^{\rm Lag}).
\end{align}
To evaluate $\tilde{f}(t_m^{\rm Lag})$ we need accurate computations of $h_\eta(t_m^{\rm Lag}/\omega)$ for $m=1,\ldots,N$. For this, for each $m$ we run a Newton iteration on \eqref{eq:SDparam} with $p=t_m^{\rm Lag}/\omega$ fixed, until the magnitude of the increment is smaller than a user-specified tolerance $\delta_{\rm fine}>0$. Typically we take $\delta_{\rm fine}$ to be  considerably smaller than the tolerance $\delta_{\rm coarse}$ used in the Newton iteration in step 4, since when carrying out quadrature we require higher accuracy in our approximation of the SD contour than is required for determining the global structure of the quasi-SD deformation in step 4. As the initial guess for the Newton method we use a piecewise linear interpolant of the points $\{(p_0,h_\eta^{(0)}),(p_1,h_\eta^{(1)}),\ldots,(p_{n_{\rm max}},h_\eta^{(n_{\rm max})})\}$ computed in step 4, where $n_{\rm max}$ denotes the total number of steps taken in the ODE solve in step 4 before the contour tracing algorithm terminated. If $p_{n_{\rm max}}<t_{N}^{\rm Lag}/\omega$ then before running the Newton iteration we first need to restart the contour tracing algorithm of step 4 to extend the SD contour until $p_{n_{\rm max}}\geq t_{N}^{\rm Lag}/\omega$, so that there are points to interpolate between.

As an alternative, one can evaluate the integral over a Type 2 contour using truncated Gauss-Legendre quadrature, as suggested in \cite{Tr:22}. 
To activate this alternative in PathFinder one should add the optional input \texttt{\textquotesingle inf quad rule\textquotesingle, \textquotesingle legendre\textquotesingle}. 
In this case we truncate the integral to 
\begin{align}
\label{eq:Type2Truncation}
\int_{\gamma}f(z)\ee^{\ii\omega g(z)}\,\dd z \approx \frac{\ee^{\ii \omega g(\eta)}}{\omega}\int_0^{P} \tilde{f}(\tilde{p})\ee^{-\tilde{p}}\,\dd \tilde{p},
\end{align}
for some $P>0$, then apply Gauss-Legendre quadrature on $[0,P]$, to obtain the approximation
\begin{align}
\label{eq:Type2GauLeg}
\int_{\gamma}f(z)\ee^{\ii\omega g(z)}\,\dd z \approx \frac{P\ee^{\ii \omega g(\eta)}}{2\omega} \sum_{m=1}^{N}w_m^{\rm Leg}\tilde{f}(z_{[0,P]}(t_m^{\rm Leg}))\ee^{-z_{[0,P]}(t_m^{\rm Leg})},
\end{align}
where we compute $h_\eta(z_{[0,P]}(t_m^{\rm Leg}/\omega))$ (which is required for the evaluation of $\tilde{f}(z_{[0,P]}(t_m^{\rm Leg}))$) by the same Newton iteration discussed above for $h_\eta(t_m^{\rm Lag}/\omega)$.   
For the truncation point $P$ we take %
\begin{align}
 \label{eq:PDef1}
P = L,
 \end{align} 
where
\begin{align}
\label{eq:LDef}
L := -\log\left(\delta_{\rm quad} M /|\ee^{\ii \omega g(\eta)}|\right),
\end{align}
which describes the point at which the magnitude of the exponential part of the integrand drops below $\delta_{\rm quad}$ times its maximum value $M$ on the quasi-SD deformation. 

\subsubsection{Evaluation of integrals over Type 3 contours}
\label{sec:Type3}
Let $\eta\in \Pexit\cup(\Pendp\setminus\Omega)$ be such that the SD contour $\gamma$ from $\eta$ leads to an entrance $z\in \Pentr$. 
In this case we apply (possibly truncated) Gauss-Legendre quadrature as in formulas \eqref{eq:Type2Truncation} and \eqref{eq:Type2GauLeg}, but now with 
\begin{align}
\label{eq:PDef2}
P = \min(p_{n_{\rm max}}/\omega,L),
\end{align}
where $p_{n_{\rm max}}$ is defined as in \S\ref{sec:SDTracing} and $L$ is defined as in \S\ref{sec:Type2}. 

In the case where the minimum is attained by $p_{n_{\rm max}}/\omega$, so that the whole contour is considered, a potential inconsistency arises, because the application of the higher accuracy Newton iteration described in \S\ref{sec:Type2} for the calculation of $h_\eta(z_{[0,P]}(t_m^{\rm Leg}/\omega))$ corresponds implicitly to a slight shifting of the endpoint of the contour $\gamma$ away from the entrance $z=h_\eta^{(n_{\rm max})}$ added to the graph $G$ in step 5. 
To avoid this inconsistency, in our implementation, in step 4, whenever the contour tracing terminates in case (i), we run a Newton iteration on the final point $h_\eta^{(n_{\rm max})}$ with the high accuracy tolerance $\delta_{\rm fine}$, before adding it to the list of entrances $\Pentr$. Note that this may mean that $h_\eta^{(n_{\rm max})}$ lies very slightly outside $\Omega$. 
\section{Theoretical results}
\label{sec:Theory}
In this section we collect some theoretical results that motivate the design of our algorithm. 
\subsection{Removal of stationary points}\label{sec:removal}

In \S\ref{sec:DefNonOsc} we described our algorithm for removing stationary points from the set $\Pstat$ when they are close. When removing stationary points and their associated non-oscillatory balls, we need to ensure that the removed stationary points still lie inside one of the remaining non-oscillatory balls, so that we don't encounter any stationary points along the trajectory in our ODE solve for the SD contour tracing (see the discussion in \S\ref{sec:SDContours} below). In this section we provide a sufficient condition on the parameter $\delta_{\rm ball}$ for this to be guaranteed. 

\begin{proposition}
Suppose that in the removal algorithm of \S\ref{sec:DefNonOsc}, $n$ stationary points have been removed from $\Pstat$. Then for any stationary point $\xi$ that was removed, there exists $\xi'\in\Pstat$ such that $|\xi-\xi'|\leq n\delta_{\rm ball} r_{\xi'}$. 
\end{proposition}
\begin{proof}
We proceed by induction on $n$. The result is trivially true for $n=0$. Assume that it is true after the removal of $n$ points, and suppose that the $(n+1)$st point is now to be removed. Let $\xi_1,\xi_2$ denote the pair of points selected as realising $\min_{\xi_1,\xi_2} d_{\xi_1,\xi_2}$, and without loss of generality suppose that $\xi_2$ is the point to be removed (so that $r_{\xi_1}\geq r_{\xi_2}$). Then $|\xi_2-\xi_1|\leq \delta_{\rm ball}r_{\xi_1}\leq (n+1) \delta_{\rm ball}r_{\xi_1}$, so the claimed property holds for $\xi_2$. Furthermore, by the inductive hypothesis, for each point	 $\xi$ previously removed, there exists $\xi'\in\Pstat$ such that $|\xi-\xi'|\leq n\delta_{\rm ball} r_{\xi'}$. If $\xi'\neq \xi_2$ then $\xi'$ will still be present in $\Pstat$ after the removal of $\xi_2$, and $|\xi-\xi'|\leq n\delta_{\rm ball} r_{\xi'}\leq (n+1)\delta_{\rm ball} r_{\xi'}$. On the other hand, if $\xi'=\xi_2$ then 
$\xi'$ will not be present in $\Pstat$ after the removal of $\xi_2$, but $\xi_1$ will be, and by the triangle inequality 
\[|\xi-\xi_1|\leq |\xi-\xi_2|+ |\xi_2-\xi_1|
\leq n\delta_{\rm ball} r_{\xi_2} + \delta_{\rm ball} r_{\xi_1}\leq (n+1)\delta_{\rm ball} r_{\xi_1},\]
completing the inductive step. 
\end{proof}

As a consequence, we obtain the following. 
\begin{corollary}
If $J>2$ and $0<\delta_{\rm ball}\leq 1/(2(J-2))$ then, after the removal algorithm has run, for every stationary point $\xi$ there exists $\xi'\in\Pstat$ such that $\xi\in \Omega_{\xi'}$ and $\dist(\xi,\partial\Omega_{\xi'})\geq r_{\xi'}/2$.
\end{corollary}

\subsection{Region of no return for SD contours}
\label{sec:NoReturn}
The following result establishes a \emph{region of no return}: once an SD contour enters this region, we can say with certainty which valley it will converge to. 
The idea behind this result is that in the region of no return the highest degree term $\alpha_J z^J$ of the polynomial $g$ is sufficiently dominant over the lower degree terms that the SD contours inside the region converge to the same valley as those corresponding to the monomial phase $\alpha_J z^J$. 

\begin{theorem}[Region of no return]
\label{thm:NoReturn}
Let $g$, $\cV$ and $R_v$, for $v\in \cV$, be as in \eqref{eq:gfull}, \eqref{eq:val} and \eqref{eq:RegOfNoReturn}. The regions $R_v$, $v\in \cV$, contain no stationary points of $g$. Furthermore, if an SD contour enters $R_v$ for some $v\in\cV$, it never leaves $R_v$.
\end{theorem}		

\begin{proof}
That $R_v$ contains no stationary points follows because if $G(r,\theta)>0$ then
\[ J|\alpha_j||z|^{J-1}>\sum_{j=1}^{J-1} j|\alpha_j||z|^{j-1} 
\geq \left| \sum_{j=1}^{J-1} j\alpha_jz^{j-1}\right|,
\]
so that $g'(z)\neq0$.

Now fix $v\in \cV$. Given $\theta'\in(0,\pi/(2J))$ and $R>0$ we define the sector
\[S_v(R,\theta'):=\{z\in\C: |\arg{z}-v|_{2\pi}<\theta' \text{ and }|z|>R\}, \]
with $|\cdot|_{2\pi}$ defined as in \eqref{eq:AbsTheta}. 
We also define the function 
\begin{align}
\label{eq:RegFunction2}
\tilde{G}(R,\theta') := |J||\alpha_J|R^{J-1}\min\left(\sin{J \theta'},\cos{J\theta'}\right) - 
\sum_{j=1}^{J-1} j|\alpha_j|R^{j-1},
\end{align}
which for each fixed $\theta'$ is a polynomial in $R$ of degree $J-1$.

We claim that if $\theta'\in(0,\pi/(2J))$ and 
$\tilde{G}(R,\theta')>0$, then if an SD contour enters $S_v(R,\theta')$ it never leaves $S_v(R,\theta')$. 
To prove this, we show that if an SD contour intersects $\partial S_v(R,\theta')$ then the direction of descent always points into $S_v(R,\theta')$. 
Since $\partial S_v(R,\theta')$ is the union of the sets
\[\{z\in\C:|\arg z-v|_{2\pi}\leq \theta'\text{ and }|z|=R\} \] 
and
\[ \{z\in\C:|\arg z-v|_{2\pi}=\theta'\text{ and }|z|\in[R,\infty)\},\]
it suffices to show that, in polar coordinates $(r,\theta)$, 
			\begin{align}
			&\Im\frac{\partial g}{\partial r}>0,\quad\text{for }|\theta-v|_{2\pi}\leq \theta' \text{ and } r=R,\label{eq:case1}\\
			&\mp\Im\frac{1}{r}\frac{\partial g}{\partial \theta}>0,\quad\text{for }\theta=v\pm \theta' \text{ (mod $2\pi$) and } r\geq R.\label{eq:case2}
			\end{align}
For \eqref{eq:case1}, let $|\theta-v|_{2\pi}\leq \theta'$. Since
			\begin{align*}
			\frac{\partial g(r\e^{\ii \theta})}{\partial r}&=\sum_{j=1}^J j\alpha_{j}\e^{\ii j\theta}r^{j-1}
			\end{align*}
and
$\Im[\alpha_J \e^{\ii J \theta}] = |\alpha_J|\cos{(J|\theta-v|_{2\pi}})$ 
(using the definition of $v$)
we have that
			\begin{align*}
			\Im\frac{\partial g(r\e^{\ii \theta})}{\partial r} &\geq J|\alpha_{J}|r^{J-1}\cos(J|\theta-v|_{2\pi}) - \sum_{j=1}^{J-1} j|\alpha_{j}|r^{j-1},
			\end{align*}
	so a sufficient condition for \eqref{eq:case1} to hold is that
\begin{align}
\label{eq:Cond1}
J|\alpha_{J}|R^{J-1}\cos(J\theta') - \sum_{j=1}^{J-1} j|\alpha_{j}|R^{j-1}>0.
			\end{align}			
For \eqref{eq:case2}, let $\theta=v\pm\theta'$ (mod $2\pi$). Since
			\begin{align*}
			\frac{1}{r}\frac{\partial g(r\e^{\ii \theta})}{\partial \theta}&=\sum_{j=1}^J \ii j\alpha_{j}\e^{\ii j\theta}r^{j-1},
			\end{align*}
and
$\Im[\ii\alpha_J \e^{\ii J \theta}]=\Im[\ii\alpha_J \e^{\ii J (v\pm\theta')}] = \mp|\alpha_J|\sin(J\theta')$ 
we have that
	\begin{align*}
	\left.\mp\Im\frac{1}{r}\frac{\partial g(r\e^{\ii \theta})}{\partial \theta}\right|_{\theta=v\pm\theta'} &\geq J|\alpha_{J}|r^{J-1}\sin(J\theta') - \sum_{j=1}^{J-1} j|\alpha_{j}|r^{j-1}=:\phi(r).
			\end{align*}
The function $\phi(r)$ has the property that if $R>0$ and $\phi(R)>0$ then $\phi(r)>0$ for all $r\geq R$. 
To see this, note that
\[ \phi(r) = r^{J-1}\Big(J|\alpha_{J}|\sin(J\theta') - \sum_{j=1}^{J-1} j|\alpha_{j}|r^{j-J}\Big),\]
and that the term in brackets is a strictly decreasing function of $r$, which tends to $-\infty$ as $r\to 0$ and to $J|\alpha_{J}|\sin(J\theta')>0$ as $r\to\infty$. 
Hence a sufficient condition for \eqref{eq:case2} is that $\phi(R)>0$, i.e.
\begin{align}
\label{eq:Cond2}
J|\alpha_{J}|R^{J-1}\cos(J\theta') - \sum_{j=1}^{J-1} j|\alpha_{j}|R^{j-1}>0.
			\end{align}	
Since the assumption $\tilde{G}(R,\theta')>0$ implies both \eqref{eq:Cond1} and \eqref{eq:Cond2}, our claim is proved. 

The statement of the theorem then follows 
by noting that the region $R_v$ is the union of all the sectors $S_v(R,\theta')$ such that $0<\theta'\leq \pi/(2J)$ and $\tilde{G}(R,\theta')>0$. 
We note that if $0<\theta'<\pi/(4J)$ then $\sin{J\theta'}<\sin{\pi/4}$, so that if $\tilde{G}(R,\theta')>0$ then $\tilde{G}(R,\pi/(4J))>0$. This implies that the union can actually be taken over $\pi/(4J)\leq \theta'<\pi/(2J)$ only, justifying the definition of the function $G$ in \eqref{eq:RegFunction}.
\end{proof}

\subsection{Quadrature error}
\label{sec:QuadError}
In \S\ref{sec:DefNonOsc} we defined the non-oscillatory region as a union of balls on which the exponential $\ee^{\ii \omega g(z)}$ undergoes a bounded number of oscillations. {In this section} we show that the definition \eqref{eq:rxiDef} strikes a balance between the accuracy of our quadrature approximations to the integrals outside and inside this region. 

{We note that, as already mentioned in \S\ref{sec:intro}, in contrast to standard NSD approximations the error in our method does not in general decay as $\omega\to\infty$. In the special case where we are tracing a single infinite SD contour from a fixed ($\omega$-independent) endpoint to a valley at infinity, without stationary points nearby, the error in our method would indeed decay as $\omega\to\infty$ in the same way as for standard NSD (e.g., \cite[Thm.~5.5]{DeHuIs:18}), up to errors introduced by our ODE for SD contour tracing. However, in the general case 
our algorithm traces steepest descent contours from points on the edges of the non-oscillatory balls, whose radii depend in a non-trivial way on $\omega$ and the distribution of stationary points, so the standard NSD theory does not apply. Also, the NSD-style Gauss-Laguerre quadrature along SD contours forms just one step in our algorithm, and other steps (such as the quadrature inside the non-oscillatory region) have their own non-trivial $\omega$-dependence. As a result, the $\omega$-dependence of the overall error is much harder to predict for our algorithm than for standard NSD in the simplest setting. 
Nonetheless, in practice we observe errors that remain bounded as $\omega\to \infty$ for fixed $N$. We shall provide some theoretical justification for this below, and back this up with extensive numerical evidence in \S\ref{sec:Numerics}.
}

\subsubsection{Quadrature in the non-oscillatory region}
\label{sec:Non-Osc}
The Type 1 straight line contour integrals between points in the non-oscillatory region are evaluated using Gauss-Legendre quadrature, as detailed in \S\ref{sec:Type1}. 
To assess the accuracy of this we note the following theorem, which is a simple consequence of the standard error analysis presented in \cite[Chap.~19]{Tr:13}. 
\begin{theorem}
\label{thm:GauLeg}
Let $z_0,z_1\in\C$. Suppose that $\gamma$ is a straight-line contour in $\C$ starting at $z_0$ and ending at $z_1$ and that there exists $\rho>0$, $C>0$ and $\xi_*\in\C$ 
such that $f$ is analytic and bounded in $z_{[z_0,z_1]}(B_\rho)$, where $B_\rho$ is a standard Bernstein ellipse (relative to $[-1,1]$) and $z_{[z_0,z_1]}$ is defined as in \eqref{eq:zDef},   
and 
\begin{align}
 \label{eq:GauLegCond}
 \omega|g(\xi_*)-g(z)|\leq C, \qquad z\in z_{[z_0,z_1]}(B_\rho).
 \end{align} 
Let $I$ and $Q$ denote the left- and right-hand sides of \eqref{eq:GauLegApplication}, respectively. 
Then, for some $\tilde{C}>0$, depending only on $\rho$, 
\begin{align}
\label{eq:GauLegBound}
\left| I-Q \right| \leq \tilde{C}|z_1-z_0|\|f\|_{L^\infty(z_{[z_0,z_1]}(B_\rho))}\ee^{-\omega \Im[g(\xi_*)]}\ee^{C}\rho^{-2N}.
\end{align}
\end{theorem}

\begin{proof}
Noting that 
\[I = \ee^{\ii \omega g(\xi_*)}\int_{\gamma}f(z)\ee^{\ii\omega (g(z)-g(\xi_*))}\,\dd z\]
and that
\[ |f(z)\ee^{\ii\omega (g(z)-g(\xi_*))}|\leq \|f\|_{L^\infty(z_{[z_0,z_1]}(B_\rho))} \ee^{C}, \qquad z\in z_{[z_0,z_1]}(B_\rho), \]
the result follows from \cite[Thm 19.3]{Tr:13}.
\end{proof}

Theorem \ref{thm:GauLeg} motivates the definition of the non-oscillatory region in \eqref{eq:rxiDef}. Indeed, if the assumptions of Theorem \ref{thm:GauLeg} hold with $\rho$ and $C$ independent of $\omega$ then the bound \eqref{eq:GauLegBound} guarantees $\omega$-independent exponential convergence for $\omega$ bounded away from zero. 
However, even when \eqref{eq:rxiDef} is satisfied, 
the relationship between $\xi_*$, $\rho$, $C$ and $\omega$ is beyond our control in general because the ellipse may extend beyond the non-oscillatory region, so that $C>C_{\rm ball}$. Thus we cannot control the factor $\ee^C$ entirely based on condition \eqref{eq:rxiDef}.

Still, the bound \eqref{eq:GauLegBound} shows that the quadrature error decreases with increasing $N$. The precise rate of decrease depends on a balance between the decay of $\rho^{-2N}$ and the growth of $\ee^C$ and $\|f\|_{L^\infty(z_{[z_0,z_1]}(B_\rho))}$ for increasing $\rho$. We quantify this in the special case of monomial phase in \S\ref{sec:Monomial}.

\subsubsection{Quadrature for the SD contours}		
\label{sec:SDContours}

For Type 2 or Type 3 integrals along SD contours we use either Gauss-Laguerre or (possibly truncated) Gauss-Legendre quadrature, as detailed in \S\ref{sec:Type2} and \S\ref{sec:Type3}. We expect these rules to converge rapidly to the true value of the integral as the number of quadrature points $N$ tends to infinity, provided that the integrand is analytic and bounded in a suitable region of the complex $\tilde{p}$ plane. 

For Gauss-Laguerre the following result appeared recently in \cite[Thm~6.3]{Wang23}.

\begin{theorem}
\label{thm:GauLag}
Suppose that $\tilde{f}$ is analytic inside and on the parabola $P_\rho:=\{z\in\C:\sqrt{-z}=\rho\}$ for some $\rho>0$, where the branch cut is along the positive real axis and $\sqrt{-z}$ is real and positive on the negative real axis, that $\tilde{f}$ grows at most algebraically as $z\to\infty$ inside the parabola, and that the integral 
\[\mathcal{K}_\rho:=\int_{P_\rho}|\ee^{-z}\sqrt{-z}\tilde{f}(z)|\,\dd z\]
is finite. 
Let $I$ and $Q$ denote the left- and right-hand sides of \eqref{eq:GauLagApplication}, respectively. 
Then
\begin{align}
\label{eq:GauLagBound}
\left| I-Q \right| \leq \mathcal{K}_\rho\frac{\ee^{-\omega\Im[ g(\eta)]}}{\omega}\ee^{-4\rho\sqrt{N}}.
\end{align} 
\end{theorem}

This result implies that our Gauss-Laguerre quadrature approximation should converge root-exponentially as $N\to\infty$, provided that $f$ is sufficiently well-behaved at infinity. 
The presence of singularities in the complex $\tilde{p}$-plane limits the size of $\rho$, and hence the convergence rate. 
We know from \eqref{eq:SDIntegralInfinite} that our integrand is singular at points $\tilde{p}\in\C$ where $g'(h_\eta(\tilde{p}/\omega))=0$, i.e. where $h_\eta(\tilde{p}/\omega)=\xi$ for some stationary point $\xi$. Since we only trace SD contours outside the non-oscillatory region (which contains the stationary points), we know that there cannot be singularities on the SD contour itself. 
If the start point $\eta$ lies on an SD contour emanating from a stationary point $\xi$ then we expect there to be a singularity in the $\tilde{p}$-plane at $\tilde{p}=\omega\Im[g(\xi)-g(\eta)]<0$. We show in \S\ref{sec:Monomial} that in the special case of monomial phase this singularity lies at $\tilde{p}=-C_{\rm ball}$, which implies root-exponential convergence independent of $\omega$ for $\omega$ bounded away from zero. Determining the locations of the other possible singularities in the complex $\tilde{p}$-plane is more challenging, since it involves study of the (multivalued) inverse of $g$. We leave further theoretical investigation of this to future work.

\subsubsection{Results for monomial phase}		
\label{sec:Monomial}

It is instructive to consider the special case of a monomial phase $g(z)=z^J$ for some $J\in\N$. In this case there is a single stationary point of order $J-1$ at $\xi=0$, and $g(0)=0$.  
Following the prescription \eqref{eq:rxiDef}, we obtain a ball radius 
\[ r_0=(C_{\rm ball}/\omega)^{1/J}.\]

We first consider a Type 1 integral in the non-oscillatory region. For simplicity we choose $f(z)\equiv 1$.  Specifically, we consider the evaluation of the integral 
\[ \int_0^{r_0\ee^{\ii \theta}} \ee^{\ii \omega g(z)}\,\dd z,\]
for some $\theta\in[0,2\pi]$. 
Taking $\xi_*=0$, we can apply Theorem \ref{thm:GauLeg} with any $\rho>1$, and the resulting scaled and translated Bernstein ellipse surrounding $[0,r_0\ee^{\ii\theta}]$ is tightly contained in the disc $|z|\leq sr_0$, where $\rho$ and $s$ are related by
\[ \rho = 2s-1+\sqrt{(2s-1)^2-1} =2s-1+2\sqrt{s^2-s}.\]
Hence condition \eqref{eq:GauLegCond}\ is satisfied, independently of $\theta$, with 
\[ C = C_{\rm ball} s^J,\]
which is independent of $\omega$ but dependent on $J$.  
When $s$ is large, we have $\rho\approx 4s$, and in this regime the error bound provided by \eqref{eq:GauLegBound} for Gauss-Legendre quadrature is approximately proportional to 
\[ (C_{\rm ball}/\omega)^{1/J}\ee^{C_{\rm ball} s^J}(4s)^{-2N}.\]
As a function of $s$, with $J$ and $N$ fixed, this quantity is minimised where its $s$-derivative vanishes, which occurs where
\[ C_{\rm ball}Js^J -2N=0,\]
i.e.\ where
\[ s = \left(\frac{2N}{C_{\rm ball}J}\right)^{1/J}.\]
Accordingly, the error bound is approximately proportional to
\[ (C_{\rm ball}/\omega)^{1/J} 16^J \left(\frac{8eN}{C_{\rm ball}J}\right)^{-2N/J}.\]
Thus we expect super-exponential convergence as $N\to\infty$ for fixed $J$. However, we expect the convergence to be slower the larger $J$ is. 

Next we consider a Type 2 integral over an SD contour, again with $f(z)\equiv 1$. Specifically, we consider the evaluation of the integral 
\[ \int_{r_0\ee^{\ii v}}^{\infty\ee^{\ii v}} \ee^{\ii \omega g(z)}\,\dd z,\]
where $v=((2j+1/2)\pi)/J$ for some $j\in\{1,\ldots,J\}$.
Following our method, the contour is parametrized by 
\[ h_\eta(p) = (r_0^J + p)^{1/J}\ee^{\ii v}, \qquad p\in [0,\infty), \]
and, recalling \eqref{eq:ODE} and \eqref{eq:SDIntegralInfinite}, after rescaling $p=\tilde{p}/\omega$ 
the integral becomes
\[ \frac{\ee^{-C_{\rm ball}}\ee^{\ii v}}{\omega^{1/J} J}\int_0^\infty (C_{\rm ball} + \tilde{p})^{1/J - 1}\ee^{-\tilde{p}}\,\dd \tilde{p}.\]
The integrand has a branch point at
\[ \tilde{p} = -C_{\rm ball},\]
but we note that the distance between the branch point and the positive real $\tilde{p}$-axis equals $C_{\rm ball}$, which is independent of both $\omega$ and $J$.
 
For truncated Gauss-Legendre the relevant theory can be found in \cite[Chap.~19]{Tr:13} (and see also \cite{Tr:22}). 
Due to the branch point at $\tilde{p}=-C_{\rm ball}$, as $N\to\infty$ we obtain exponential convergence to the integral over the interval $[0,P]$, where $P$ is given by either \eqref{eq:PDef1} or \eqref{eq:PDef2}. 
In the case where $P=L$, by the definition of $L$ in \eqref{eq:LDef}, we expect the truncation error to have relative order $\delta_{\rm quad}$.

\subsubsection{Number and distribution of quadrature points}
\label{sec:NumberOfQuadPoints}

PathFinder uses a fixed number $N$ of quadrature points on each contributing contour, and that number is the same both for integrals within and outside the non-oscillatory region, i.e., for Gauss--Legendre and Gauss--Laguerre quadrature. Thus, increasing the single parameter $N$ provides a way of uniformly improving accuracy.

The theoretical results in this section (specifically, Theorems \ref{thm:GauLeg} and \ref{thm:GauLag}) imply that the precise rate of improvement with respect to $N$ depends on the type of integral being approximated. They suggest even that a different strategy for the distribution of quadrature points may be superior. Indeed, exponential convergence of Gauss--Legendre 
for Type 1 integrals in the non-oscillatory region is not balanced with root-exponential convergence of Gauss--Laguerre for Type 2 integrals outside. Similarly, convergence rates of Gauss-Laguerre and truncated Gauss-Legendre outside the non-oscillatory region are different. Our choice of a fixed parameter $N$ is inspired on the one hand by simplicity, and on the other hand by the lack of robust methods to optimize parameters in alternative schemes. For example, we have shown in \S\ref{sec:Monomial} that the convergence rate of Gauss-Legendre for Type 1 integrals may depend on the order of nearby stationary points. While this can be quantified precisely for the case of monomial phase, it is not at all clear how to generalise this analysis when a cluster of multiple stationary points is present. Hence, 
stationary point order 
is a quantity that we deliberately do not explicitly compute, estimate or rely on in any way. Implicitly, of course, it plays a big role, and it does so mainly via the definition of the ball of the radius in \eqref{eq:rxiDef}. 

The main practical benefit of the theoretical analysis of quadrature error in this section is the guarantee that $N$ is a robust parameter for improving accuracy. Concerning possible future improvements, rather than attempting to optimize the quadrature point distribution a priori, we believe a more promising development would be the ability to invoke standard adaptive quadrature schemes along the contours for a given function $f$. 
However, it should be borne in mind that quadrature forms just one step in our algorithm, and that the other steps (particularly the SD path tracing) incur a non-negligible cost overhead, that should also be considered when trying to further optimize performance. 

\section{Further implementation aspects}
\label{sec:Details}

In this section we discuss some additional aspects of the implementation of our algorithm in PathFinder. %

\subsection{Default parameter values}
\label{sec:Defaults}

In Table \ref{tab:Parameters} we list the user-specified parameters in our algorithm, along with the default values used in all our numerical results in \S\ref{sec:Numerics}. These were determined as the result of extensive numerical experiments on a range of examples, not detailed here. Instructions on how to adjust these parameters away from their default values can be found at \textsf{github.com/AndrewGibbs/PathFinder}.

\begin{table}[t!]
\centering
\begin{tabular}{|c|c|c|c|}
\hline
Parameter & Domain & Meaning & Default \\
\hline
$C_{\rm ball}$ & $(0,\infty)$ & 
Governs maximum number of 
& $2\pi$
\\ & & 
oscillations across each non-oscillatory 
& \\ & & 
ball (and hence the ball radius) 
& \\
\hline
$N_{\rm ball}$ & $\N$ & 
Number of rays used 
& 16
\\ & & 
when determining the ball radius 
& \\
\hline
$\delta_{\rm ball}$ & $(0,1)$ & 
Governs when overlapping balls 
& $10^{-3}/(2\max(J-2,1))$ %
\\ & & 
should be amalgamated 
& \\
\hline
$\delta_{\rm ODE}$ & $(0,1)$ & 
Governs the local step size in 
& $0.1$
\\ & & 
the ODE solver for SD path tracing 
& \\
\hline
$\delta_{\rm coarse}$ & $(0,1)$ & 
Tolerance for the increment in the 
& $10^{-2}$
\\ & & 
Newton iteration in the SD path tracing 
& \\
\hline
$\delta_{\rm fine}$ & $(0,1)$ & 
Tolerance for the increment in the 
& $10^{-13}$
\\ ($<\delta_{\rm coarse}$) & & 
Newton iteration in the quadrature 
& \\
\hline
$\delta_{\rm quad}$ & $(0,1)$ & 
Governs when the contribution 
& $10^{-16}$
\\ & & 
from an integral on the quasi-SD 
& \\ & & 
deformation is computed &\\
\hline
$N$ & $\N$ & 
Number of quadrature points to use 
& no default
\\ & & 
in each integral evaluated in step 76
& \\
\hline
\end{tabular}
\caption{User-specified parameters and their default values in PathFinder.}
\label{tab:Parameters}
\end{table}

\subsection{Small $\omega$} 
\label{sec:Smallomega}

While our algorithm is geared towards the case where $\omega$ is moderate or large, we make a brief comment on the case where $\omega$ is small. 
If $\Gamma$ is infinite then the integral \eqref{eq:I} typically diverges  for $\omega= 0$. However, if $\Gamma$ is finite then the integral converges for $\omega=0$ and for small enough $\omega$ it is non-oscillatory. In PathFinder we detect and deal with this case in the following way. If both endpoints are finite, then before starting step 1 of the algorithm we construct non-oscillatory balls around the endpoints (using the process in \S\ref{sec:DefNonOsc}) and check whether the balls intersect non-trivially. If so, we apply standard Gauss-Legendre quadrature to evaluate \eqref{eq:I}; if not, the balls are discarded and we proceed with the rest of the algorithm.

\subsection{The case $J=1$}
\label{sec:Jequalsone}

In the case $J=1$ (linear phase) there are no stationary points, and our algorithm simplifies dramatically. Furthermore, the SD contours are simply parallel straight lines in the direction of the single valley at angle $\pi/2-\arg(\alpha_1)$, and there is no need to trace them numerically. Hence when $J=1$ PathFinder skips the ODE contour tracing step and exploits the exact characterization of the SD contours mentioned above.

\subsection{Specifying infinite endpoints} 
\label{sec:InfiniteEndpoints}

In the description of our algorithm in \S\ref{sec:Algorithm} we made the assumption that any infinite endpoint of the contour $\Gamma$ should be at a valley $v\in \cV$. 
PathFinder is actually more flexible than this. The user is permitted to specify an infinite endpoint at any $\theta\in [v-\pi/(2J),v+\pi/(2J)]$ and the code will automatically adjust this to equal $v$. 
The case $\theta = v\pm\pi/(2J)$ is delicate because the highest order term in the phase does not provide exponential decay along the contour. Nonetheless, we include it, because in applications one often encounters this case, with the integral converging conditionally (under appropriate assumptions on $f$) and the contour deformation to $v$ being justified by Jordan's Lemma.

\section{Numerical results}		
\label{sec:Numerics}

In this section we present numerical results illustrating the performance of our algorithm and its implementation in PathFinder.  
All results in this section were produced using PathFinder Version 1.0 \cite{PathFinder}.

\subsection{A ``generic'' example}
\label{sec:DigitsOfPiExample}

\begin{figure}[t!]
\def\h{60mm}
	\centering
\subfloat[$\omega = 0.01$]{
	\includegraphics[height=\h]{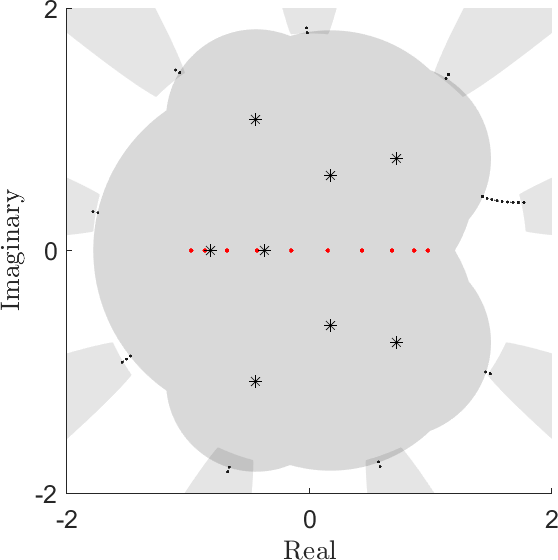}}
\subfloat[$\omega = 1$]{
	\includegraphics[height=\h]{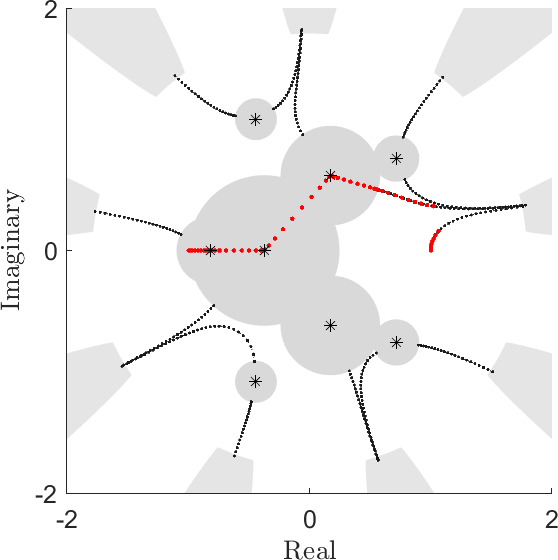}}\\
	\subfloat[$\omega = 5$]{
	\includegraphics[height=\h]{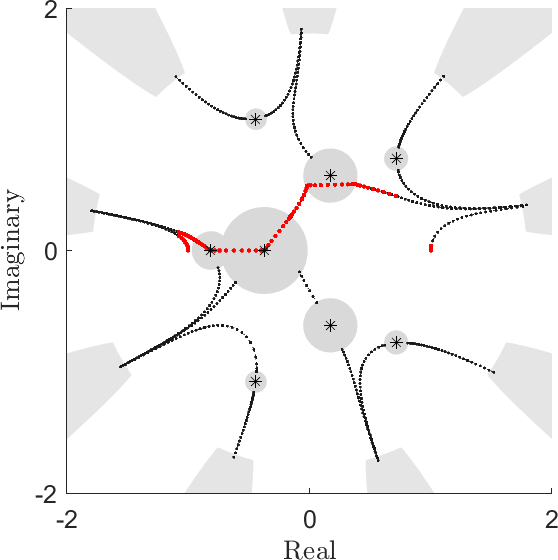}}
\subfloat[$\omega = 50$]{
	\includegraphics[height=\h]{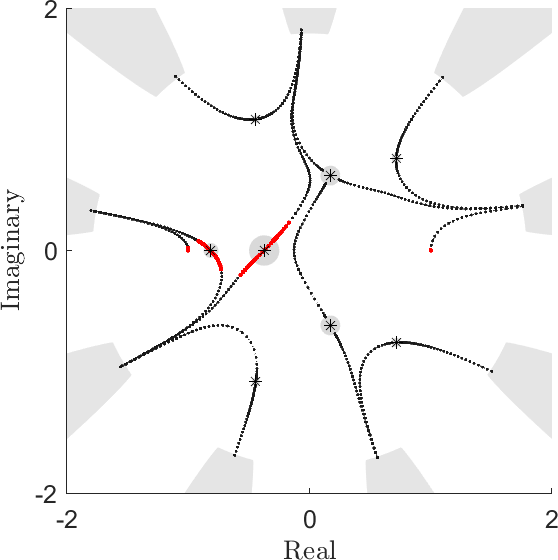}}
	\caption{PathFinder output (cf.~Figure \ref{fig:intro_eg}) with $N=10$ for the approximation of \eqref{eq:DigitsOfPiExample}.}
	\label{fig:DigitsOfPiDeformation}
\end{figure}

We begin by illustrating the performance of PathFinder on the integral
\begin{align}
\label{eq:DigitsOfPiExample}
I = \int_{-1}^1 (2z^4+7z^3+z^2+8z+2) \ee^{\ii\omega(3z^9 + z^8 + 4 z^7 + z^6 + 5 z^5 + 9 z^4 + 2 z^3 + 
      6 z^2 + 5 z + 3)}\,\dd z,
\end{align}
where, to convey the message that our approach is applicable to truly ``generic'' amplitudes and polynomial phase functions, the coefficients of $f$ and $g$ are chosen to be the first 5 digits of $\ee$ and the first 10 digits of $\pi$, respectively. 
This can be approximated by PathFinder via the Matlab code (cf.\ \eqref{eq:PathFinderCode})
\SaveVerb{VerbA}|PathFinder(-1,1,@(z) 2*z.^4+7*z.^3+z.^2+8*z+2,...|%
\SaveVerb{VerbB}|[3 1 4 1 5 9 2 6 5 3],omega,N)|%
\begin{align*}
  &\UseVerb{VerbA} \\
  &\qquad\UseVerb{VerbB}
\end{align*}
In Figure \ref{fig:DigitsOfPiDeformation} we plot the quasi-SD deformations and quadrature point distributions (using the PathFinder \texttt{\textquotesingle plot\textquotesingle} option) for \eqref{eq:DigitsOfPiExample} for $\omega\in\{0.01,1,5,50\}$ and $N=10$. As explained in \S\ref{sec:DefNonOsc}, for smaller $\omega$ the non-oscillatory balls are larger, and can overlap, while for larger $\omega$ they shrink around the stationary points. 
In more detail, in Figure \ref{fig:DigitsOfPiDeformation}(a) ($\omega=0.01$), $\omega$ is small enough that both endpoints are inside the same non-oscillatory ball. Hence the integral is treated as non-oscillatory and is approximated by Gauss-Legendre quadrature along a single straight-line contour. 
In Figure \ref{fig:DigitsOfPiDeformation}(b) ($\omega=1$), $\omega$ is still small enough that many of the balls overlap, and the quasi-SD deformation comprises two SD contours (one from an exit and one from an endpoint) plus four straight-line contours in the non-oscillatory region. 
In Figure \ref{fig:DigitsOfPiDeformation}(c) ($\omega=5$), $\omega$ is large enough that only two balls overlap, and the quasi-SD deformation comprises five SD contours (two from endpoints, two from exits to valleys, and one from an exit to an entrance), plus four straight-line contours in the non-oscillatory region. 
Finally, in Figure \ref{fig:DigitsOfPiDeformation}(d) ($\omega=50$), $\omega$ is so large that none of the balls overlap, and the quasi-SD deformation comprises eight contributing SD contours (two from endpoints and six from exits to valleys), plus three straight-line contours in the non-oscillatory region. However, in this case the two SD contours and one straight-line contour associated with the stationary point near $0.2+0.5\ii$ are judged to make a negligible contribution to the integral, so are not assigned any quadrature points. 
We emphasize that this intricate behaviour is fully automated, with no expert input required from the user. 

\begin{figure}[t!]
	\centering
\subfloat[Accuracy]{
	\includegraphics[width=.5\linewidth]{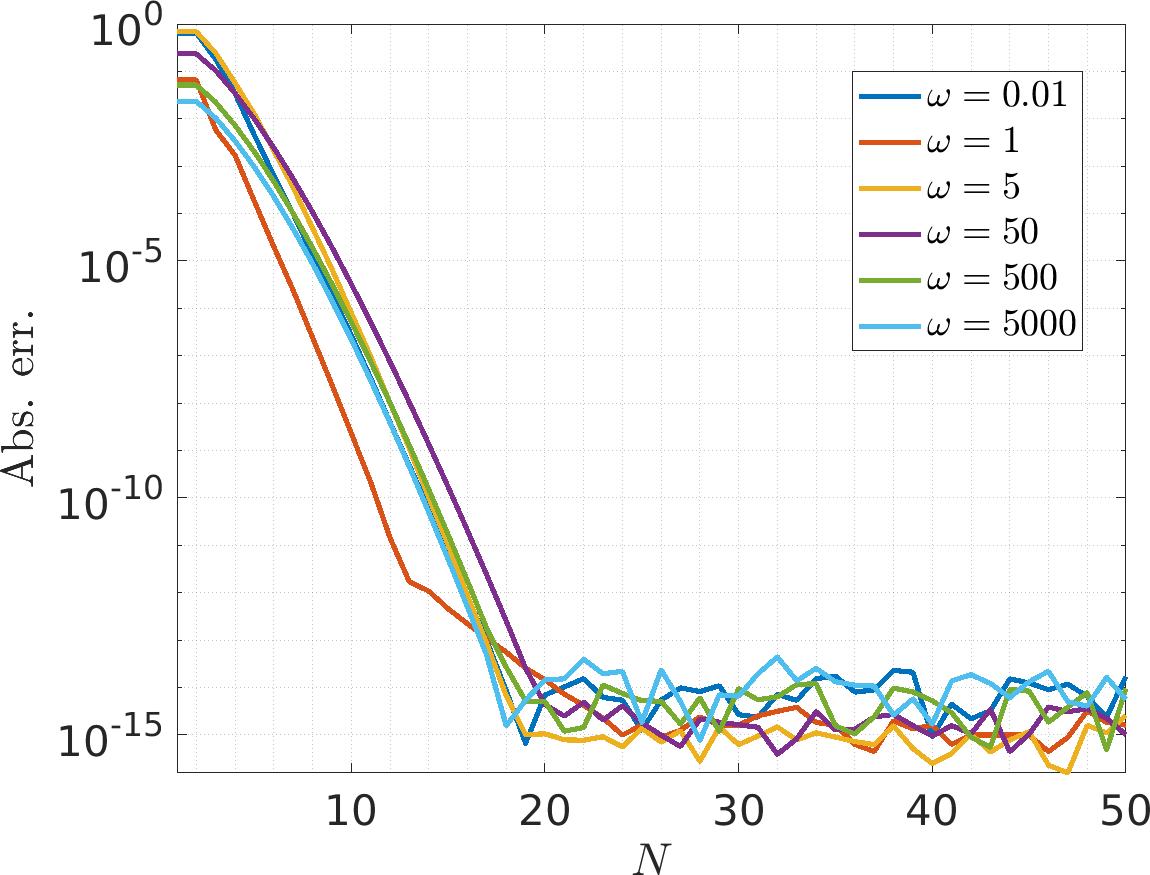}}
\subfloat[Timings]{
	\includegraphics[width=.5\linewidth]{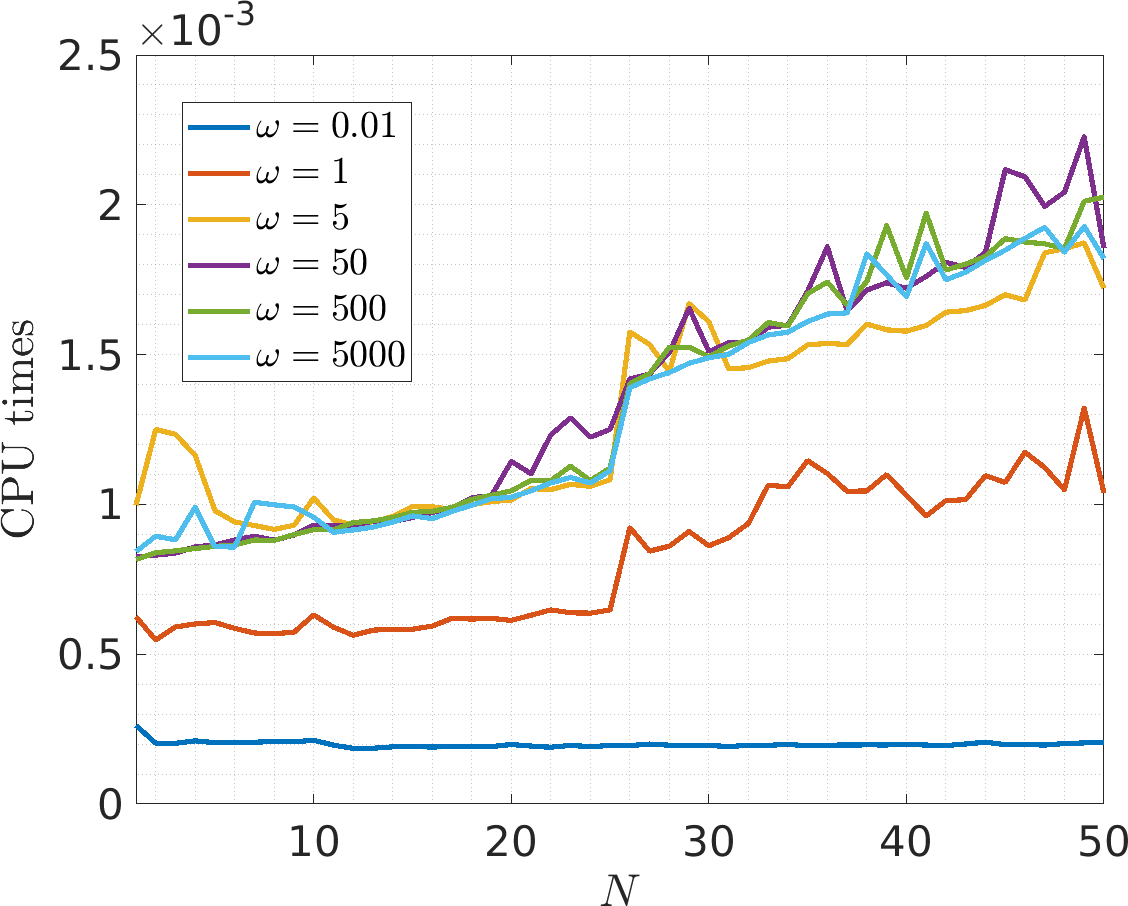}}\\
\caption{Accuracy (a) and timings (b) of the PathFinder approximation of \eqref{eq:DigitsOfPiExample}.}
	\label{fig:DigitsOfPi}
\end{figure}

In Figure \ref{fig:DigitsOfPi}(a) we plot the error in the PathFinder approximation of  \eqref{eq:DigitsOfPiExample}, compared to reference values computed using the Julia \texttt{QuadGK} package 
when $\omega<500$, and using PathFinder with $N=500$ when $\omega\geq 500$. 
For fixed $\omega$ we observe rapid convergence as $N\to\infty$, at a rate that appears independent of $\omega$. In Figure \ref{fig:DigitsOfPi}(b) we show the associated computation times, which remain bounded as $\omega$ increases. 
\subsection{Coalescence and the Airy function}
\label{sec:Airy}
The canonical example of an integral with two coalescing stationary points is provided by the integral representation for the Airy function, viz.\ (see \cite[9.5.4]{DLMF}) 
\begin{align}
\label{eq:AiryRep}
\Ai(x) = \int_{\infty \ee^{-\ii\pi/3}}^{\infty \ee^{\ii\pi/3}} \ee^{z^3/3-xz}\,\dd z= \int_{\infty \ee^{-\ii\pi/3}}^{\infty \ee^{\ii\pi/3}} \ee^{\ii(-\ii(z^3/3-xz))}\,\dd z, \qquad x\in \C, 
\end{align}
which 
is of the form \eqref{eq:I} with $f\equiv 1$, $\omega=1$ and $g(z;x)=-\ii(z^3/3- xz)$.  
Up to a change of variable this is the same example for which, as mentioned in \S\ref{sec:intro}, a bespoke, complex Gaussian quadrature rule was developed in \cite{HuJuLe:19}. 
$\Ai$ can be approximated by PathFinder %
via the Matlab code %
\begin{align*}
&\texttt{Ai = @(x) PathFinder(-pi/3,pi/3,[],...}\\
&\,\,\qquad \texttt{-1i*[1/3 0 -x 0],1,N,\textquotesingle infcontour\textquotesingle,[true true])}
\end{align*}
where the input \texttt{[]} for \texttt{f} indicates that $f\equiv 1$. 

\begin{figure}[t!]
\def\h{40mm}
	\centering
\subfloat[$x=-5$]{
	\includegraphics[height=\h]{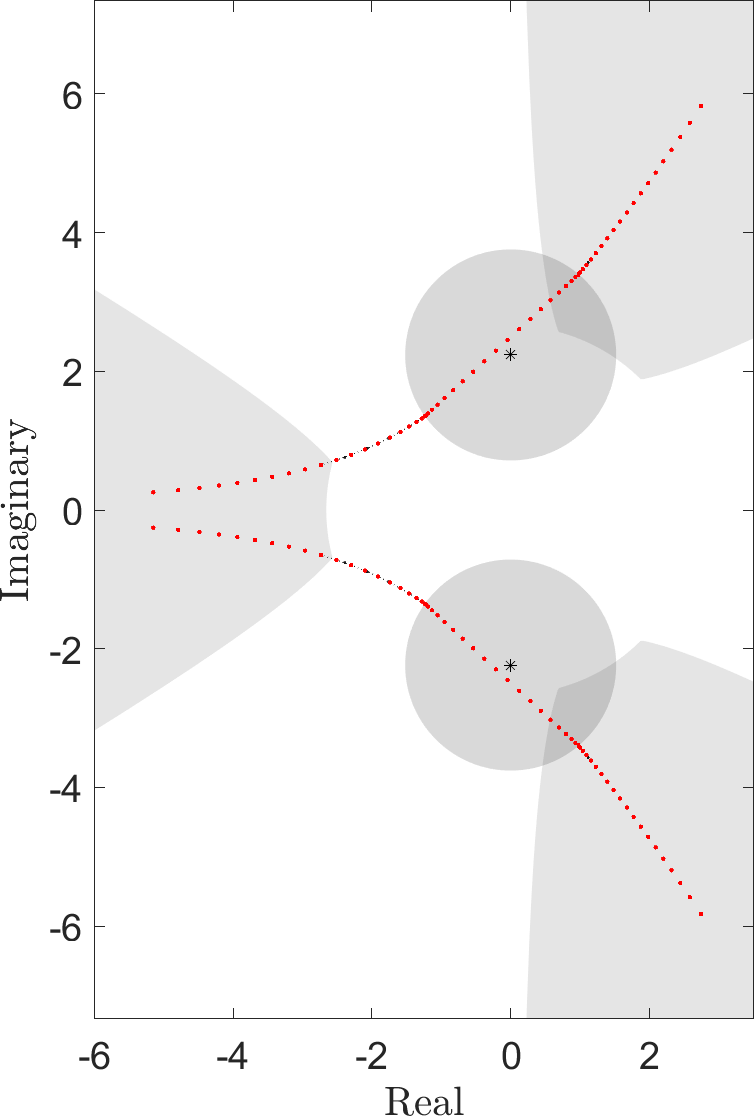}}
\subfloat[$x=-1$]{
	\includegraphics[height=\h]{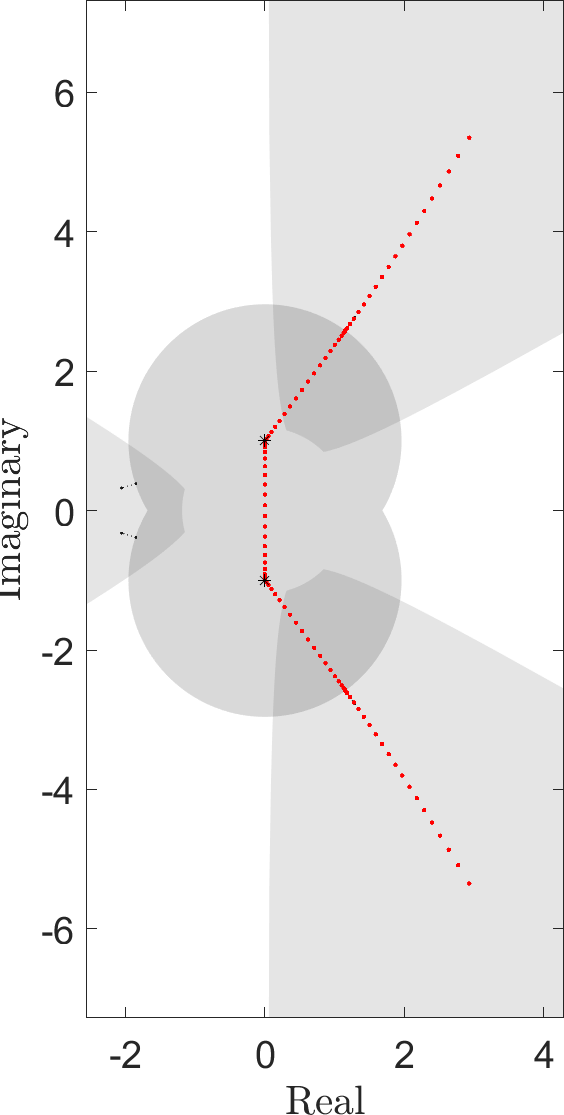}}
\subfloat[$x=-0.5$]{
	\includegraphics[height=\h]{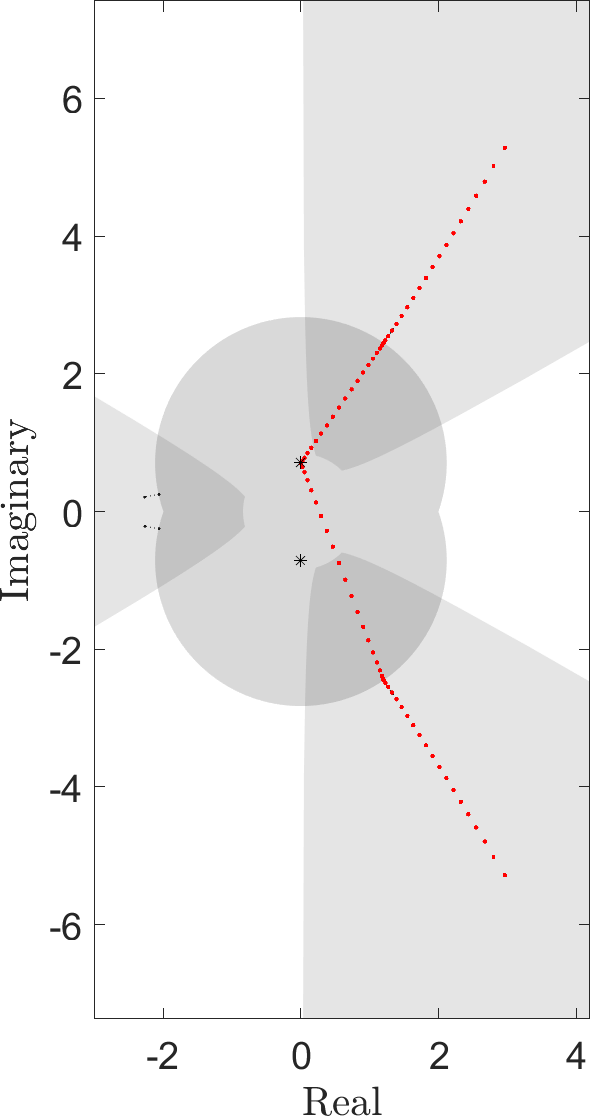}}
	\subfloat[$x=0$]{
	\includegraphics[height=\h]{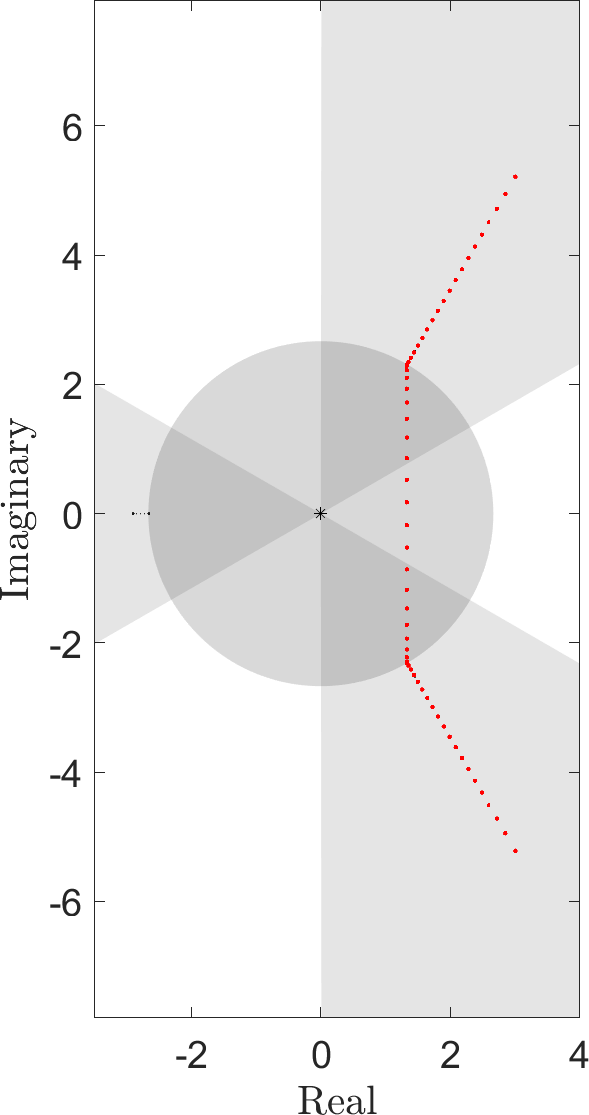}}
	\subfloat[$x=5$]{
	\includegraphics[height=\h]{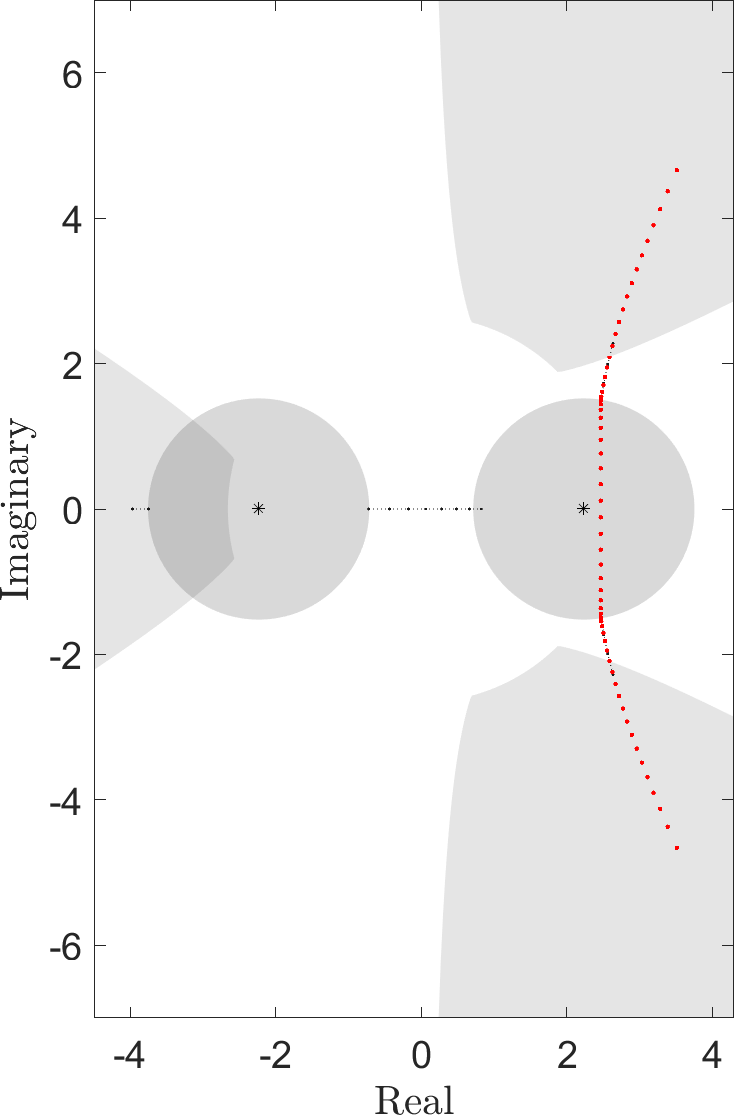}}
	\caption{PathFinder output 
	with $N=20$ for the approximation of $\Ai(x)$ via \eqref{eq:AiryRep} at various $x$, showing the stationary point coalescence at $x=0$.}
	\label{fig:airydeformation}
\end{figure}

In Figure \ref{fig:airydeformation} we plot the quasi-SD deformations, along with the distribution of quadrature points for $N=20$, for the evaluation of $\Ai(x)$ at $x\in\{-5,-1,-0.5,0,5\}$. 
Here one observes in detail how our algorithm deals with stationary point coalescence, as the non-oscillatory balls overlap, merge, then split. 
In Figure \ref{fig:airydeformation}(a) the quasi-SD deformation comprises four SD contours from exits, plus two straight-line contours inside balls (which do not go via stationary points). 
In Figure \ref{fig:airydeformation}(b) the balls overlap and this changes to two SD contours from exits plus three straight-line contours inside balls (which go via both stationary points). 
In Figure \ref{fig:airydeformation}(c) the balls overlap enough that both stationary points are contained in both balls, so we get two SD contours from exits plus just two straight-line contours inside balls (which go via only one of the stationary points). 
In Figure \ref{fig:airydeformation}(d) the balls have merged completely and in addition to the two SD contours from exits there is just one straight-line contour inside a ball (which does not go via the stationary point). 
In Figure \ref{fig:airydeformation}(e) the balls have split again, but we see the same deformation structure as in Figure \ref{fig:airydeformation}(d).
Again, we emphasize that these calculations are fully automated. 

\begin{figure}[t!]
	\centering
	\includegraphics[width=0.7\linewidth]{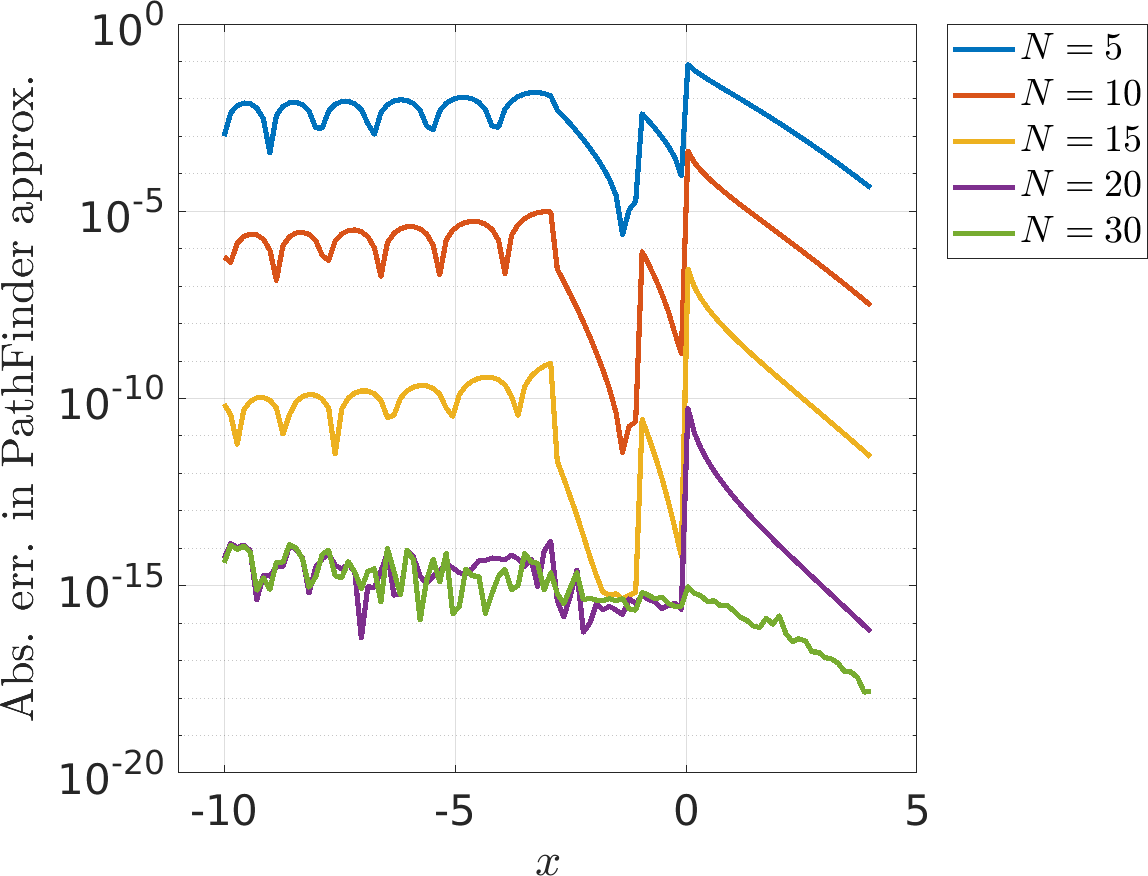}
	\caption{Accuracy of PathFinder approximation of $\Ai(x)$ for different $N$.}
	\label{fig:airyerror}
\end{figure}

In Figure \ref{fig:airyerror} we show the accuracy of the PathFinder approximation for this example as a function of $x\in[-10,4]$, for different $N$. Our reference is the built-in Matlab command \texttt{airy}. %
We note that between $x=-3$ and $x=0$ the error for the smaller values of $N$ undergoes some jumps. These are due to the fact that near stationary point coalescence  
the topology of the quasi-SD deformation, the number of contours constituting it, and hence the total number of quadrature points along it (recall \eqref{eq:Ntot}), all change discontinuously as a function of $x$ (as illustrated in Figure \ref{fig:airydeformation}).  However, as $N$ increases we see a clear, approximately exponential decrease in the error, and, although the rate of decrease depends slightly on $x$ (because of the factors mentioned above), for $N=30$ we achieve approximately $10^{-13}$ error uniformly across the interval.

\subsection{A high order stationary point - comparison with Mathematica's implementation of Levin quadrature}

\begin{figure}[t!]
	\centering
	\includegraphics[width=.8\linewidth]{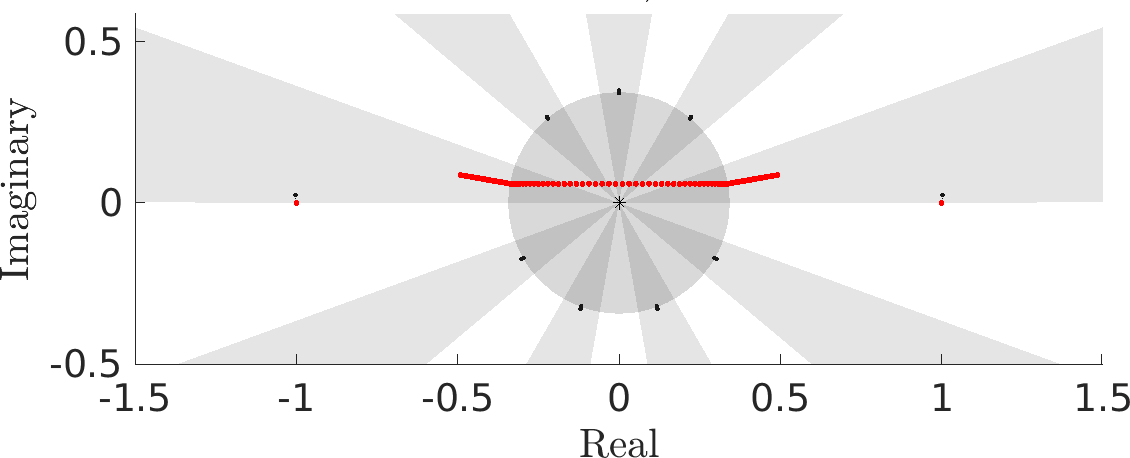}
	\caption{PathFinder output for \eqref{eq:montest} with $\omega=100,000$ and $N=50$.}	
	\label{fig:mathematicamonomialdeformation}
\end{figure}

We now consider the integral
\begin{equation}\label{eq:montest}
I = \int_{-1}^1\sin(z)\ee^{\ii\omega z^9}\dd z,
\end{equation}
which has a stationary point of order 8 at the origin. The integral \eqref{eq:montest} can be approximated by PathFinder via the command
\[
\texttt{PathFinder(-1,1,@(z) sin(z),[1 0 0 0 0 0 0 0 0 0],omega,N)}
\]
Figure \ref{fig:mathematicamonomialdeformation} shows the quasi-SD deformation and quadrature point distribution obtained by PathFinder for $\omega=100,000$ and $N=50$. There are small contributions from the endpoints, but the main contribution comes from the ball containing the stationary point. 

In the Mathematica documentation \cite[pp75-86]{MathematicaHandbook}, it is stated that oscillatory integrals with monomial phase functions such as \eqref{eq:montest} can be 
evaluated efficiently using the built-in Mathematica function \texttt{NIntegrate}, via
its implementation of Levin quadrature (which is described, e.g.,\ in \cite[\S3.3]{DeHuIs:18}). To do this one can use the Mathematica command:
\SaveVerb{VerbA}|NIntegrate[Sin[x]Exp[omega*I*x^9],{x,-1,1},|%
\SaveVerb{VerbB}|Method->{"LevinRule","Kernel"->Exp[omega*I*x^9]}]|%
\begin{align*}
  &\UseVerb{VerbA} \\
  &\qquad\UseVerb{VerbB}
\end{align*}

\begin{figure}[t!]
	\centering
	\subfloat[Accuracy (compared to Mathematica)]{\includegraphics[height=52mm]{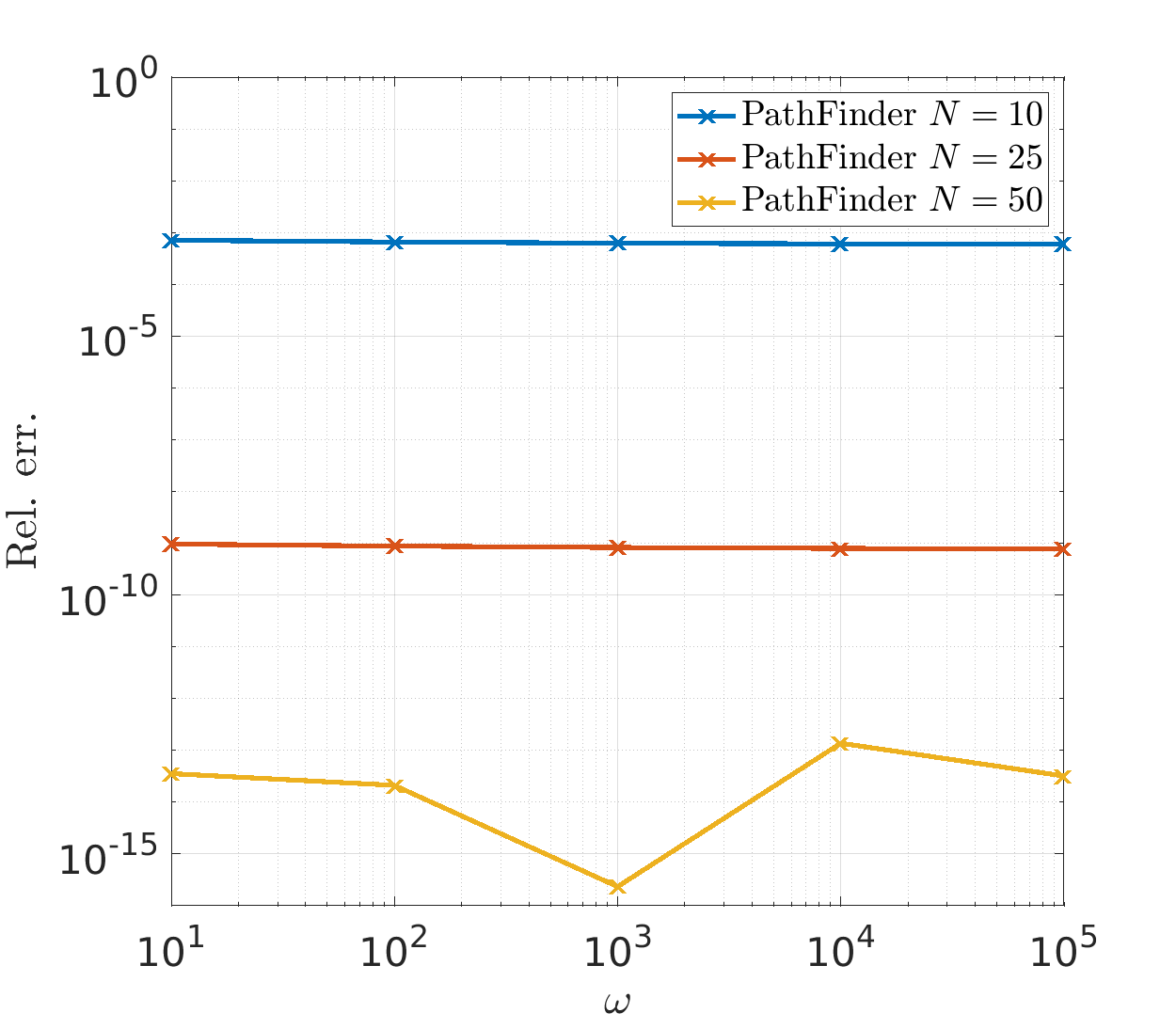}}
		\subfloat[Timings]{\includegraphics[height=52mm]{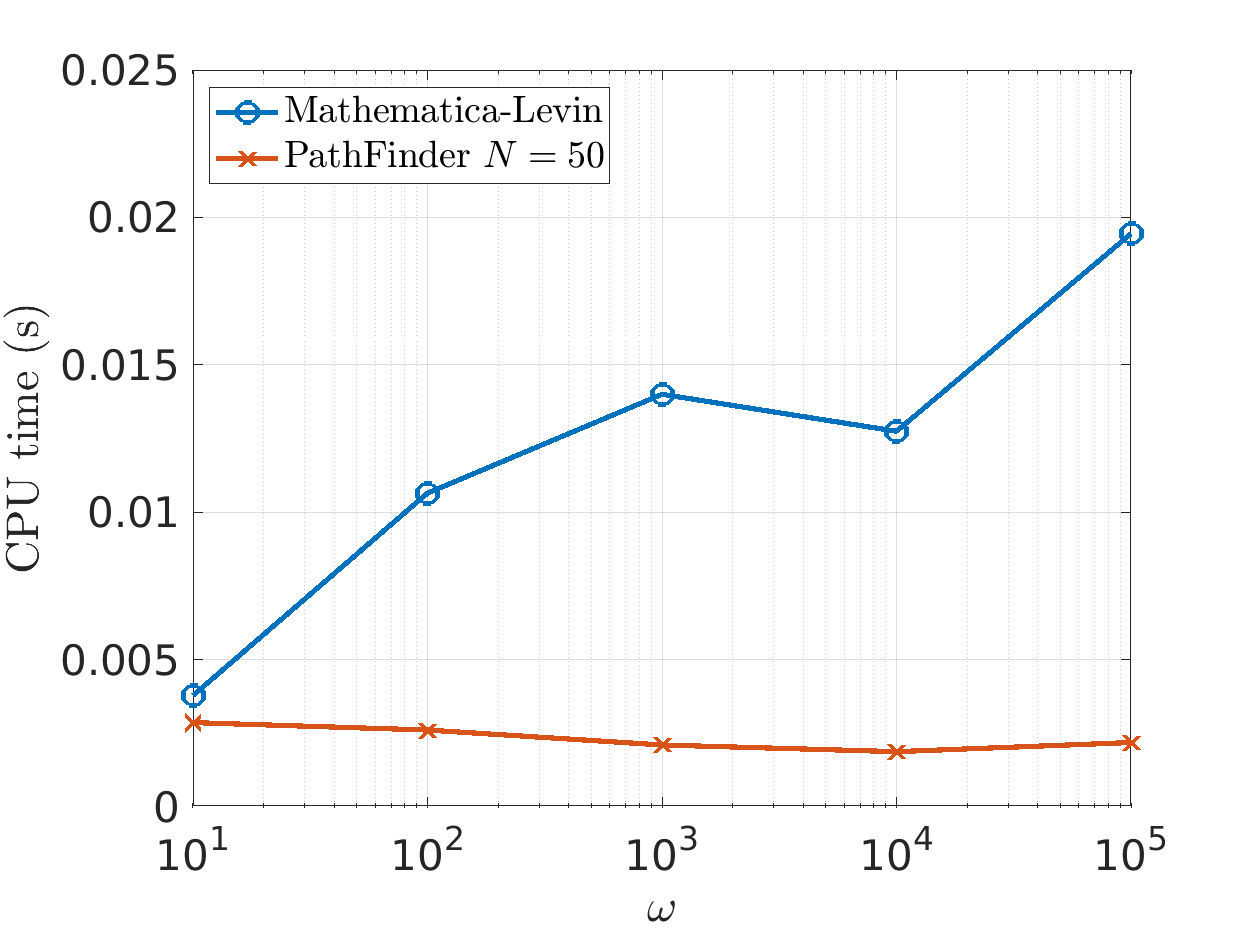}}
	\caption{Accuracy (a) and timings (b) of the PathFinder approximation of \eqref{eq:montest}, compared to Mathematica's \texttt{NIntegrate} command.}
	\label{fig:mathematicamonomial}
\end{figure}
In Figure \ref{fig:mathematicamonomial}(a) we show a plot of the relative accuracy of our PathFinder approximation, compared to the Mathematica approximation (using the default settings), as a function of $\omega$, for different $N$ values. For all three $N$ values the accuracy of our approximation is approximately uniform in $\omega$, and for $N=50$ our approximation agrees with Mathematica's to approx $13$ digits. 
In Figure \ref{fig:mathematicamonomial}(b) we report the corresponding computation times (averaged over 100 identical runs) for the Mathematica routine and for the PathFinder approximation with $N=50$. 
These results were obtained on a laptop (i7-1185G7, 32GB RAM) running Mathematica v13.0 and Matlab v2021b. The results suggest that PathFinder is highly competitive with Mathematica for this problem, especially for large $\omega$.

\subsection{Coalescence to a high order stationary point}
\label{sec:Coalescence}

\begin{figure}[t!]
	\centering
\subfloat[$r=0.001$]{
	\includegraphics[height=30mm]{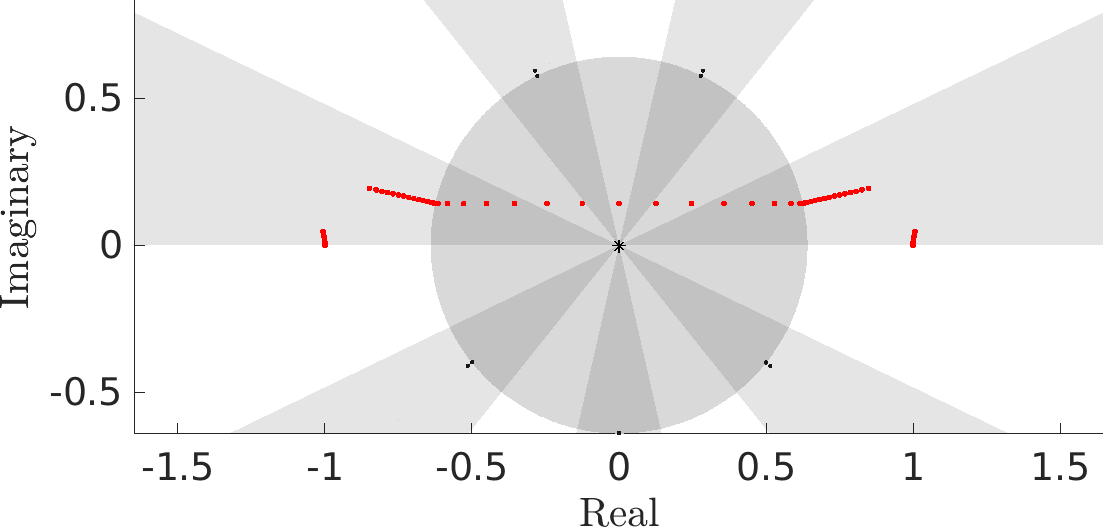}}
\subfloat[$r=0.35$]{
	\includegraphics[height=34mm]{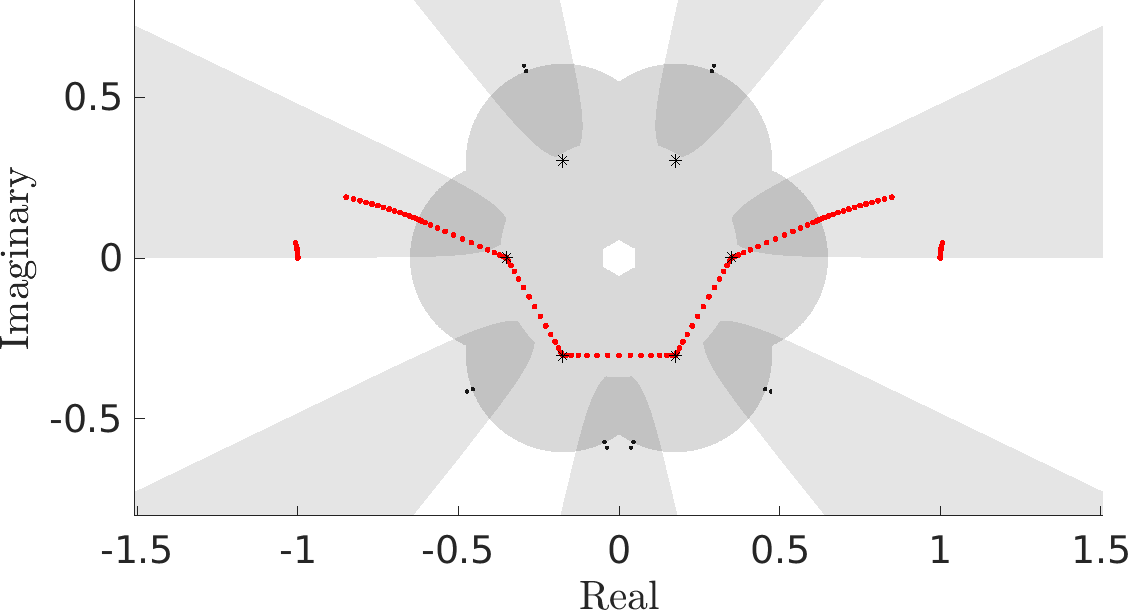}}\\
	\subfloat[$r=0.5$]{
	\includegraphics[height=45mm]{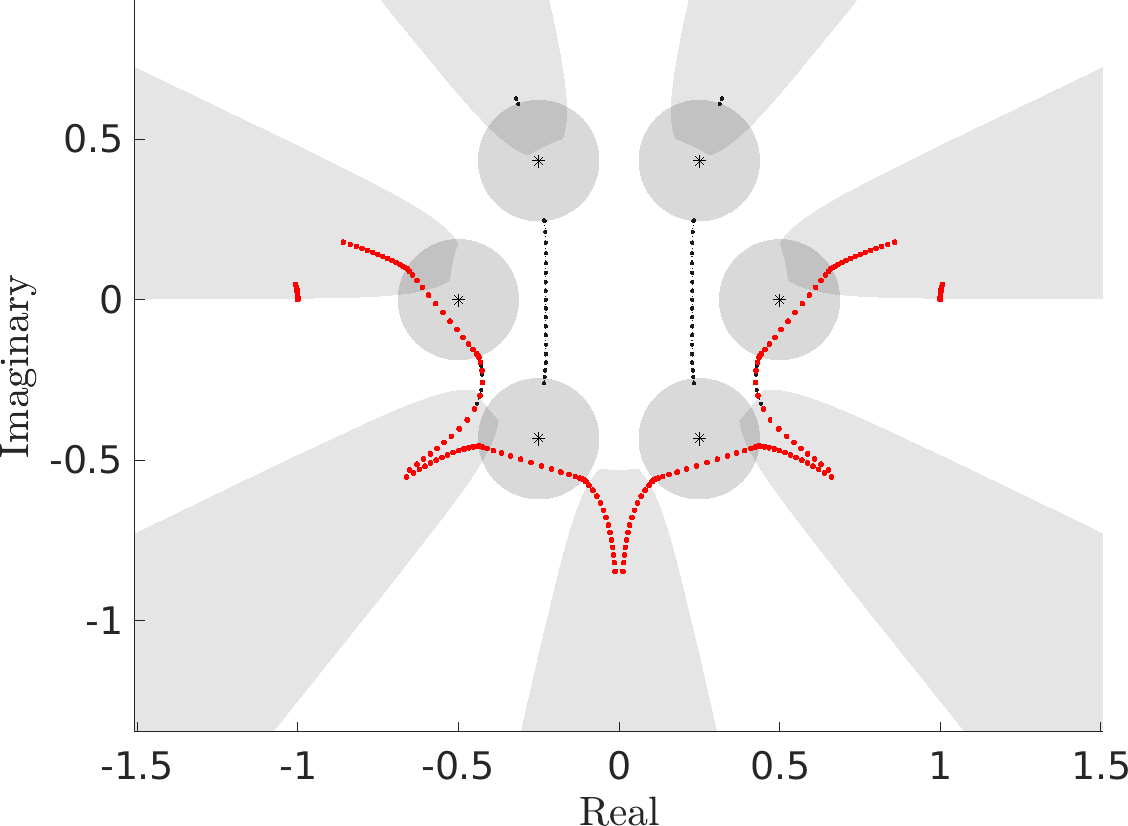}}
\subfloat[$r=1$]{
	\includegraphics[height=45mm]{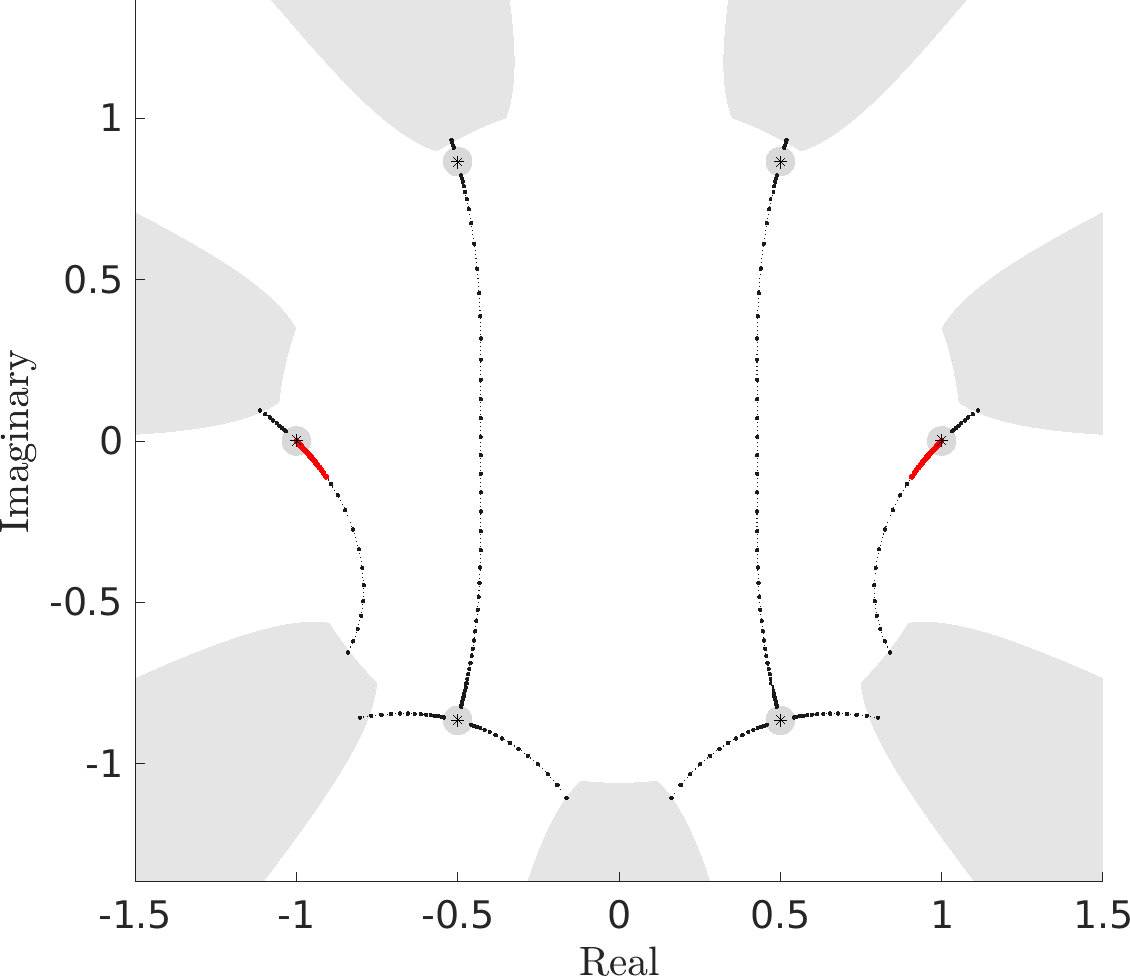}}
	\caption{PathFinder output 
for the approximation of \eqref{eq:Coalescence} with $\omega=1000$ and $N=15$.}
	\label{fig:Coalescence}
\end{figure}

We now investigate the robustness of our algorithm in the presence of a large number of coalescing stationary points. Specifically, we consider the integral
\begin{align}
\label{eq:Coalescence}
\int_{-1}^1 \ee^{\ii \omega(z^7/7-r^6 z)}\,\dd z,
\end{align}
where $r\geq 0$ is a parameter controlling the coalescence. For $r>0$ there are 6 stationary points with $|\xi|=r$, namely the solutions of $\xi^6 = r^6$, and for $r=0$ there is a single stationary point of order $6$. To evaluate this integral in PathFinder for a given $r$, one can use the command
\SaveVerb{VerbA}|PathFinder(-1,1,[],[1/7 0 0 0 0 0 -r^6 0],omega,N)|%
\begin{align*}
\UseVerb{VerbA}
\end{align*}
In Figure \ref{fig:Coalescence} we plot the quasi-SD deformations and quadrature point distributions for a some different $r$ values, showing how the balls first intersect and then merge as $r\to0$. 

\begin{figure}[t!]
	\centering
	\subfloat[Convergence]{
	\includegraphics[width=0.5\linewidth]{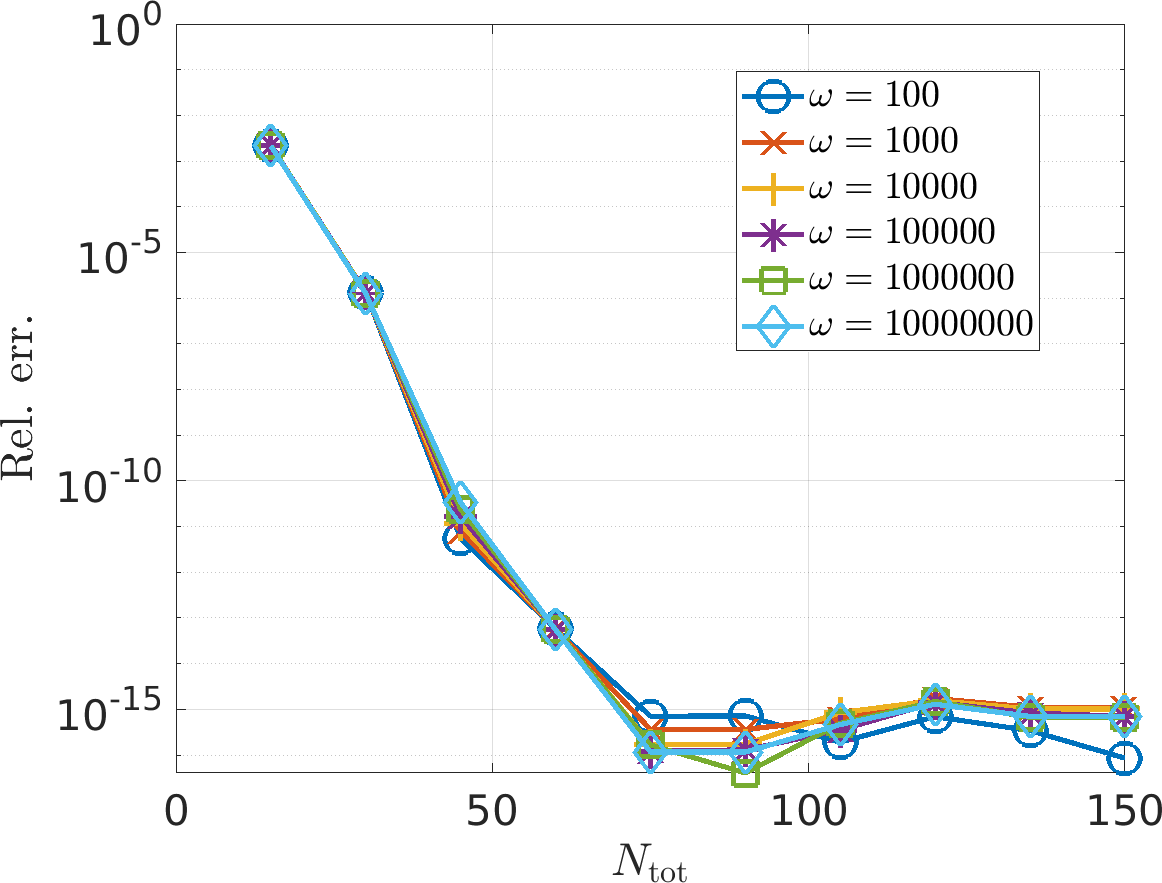}}
		\subfloat[Timings]{
	\includegraphics[width=0.5\linewidth]{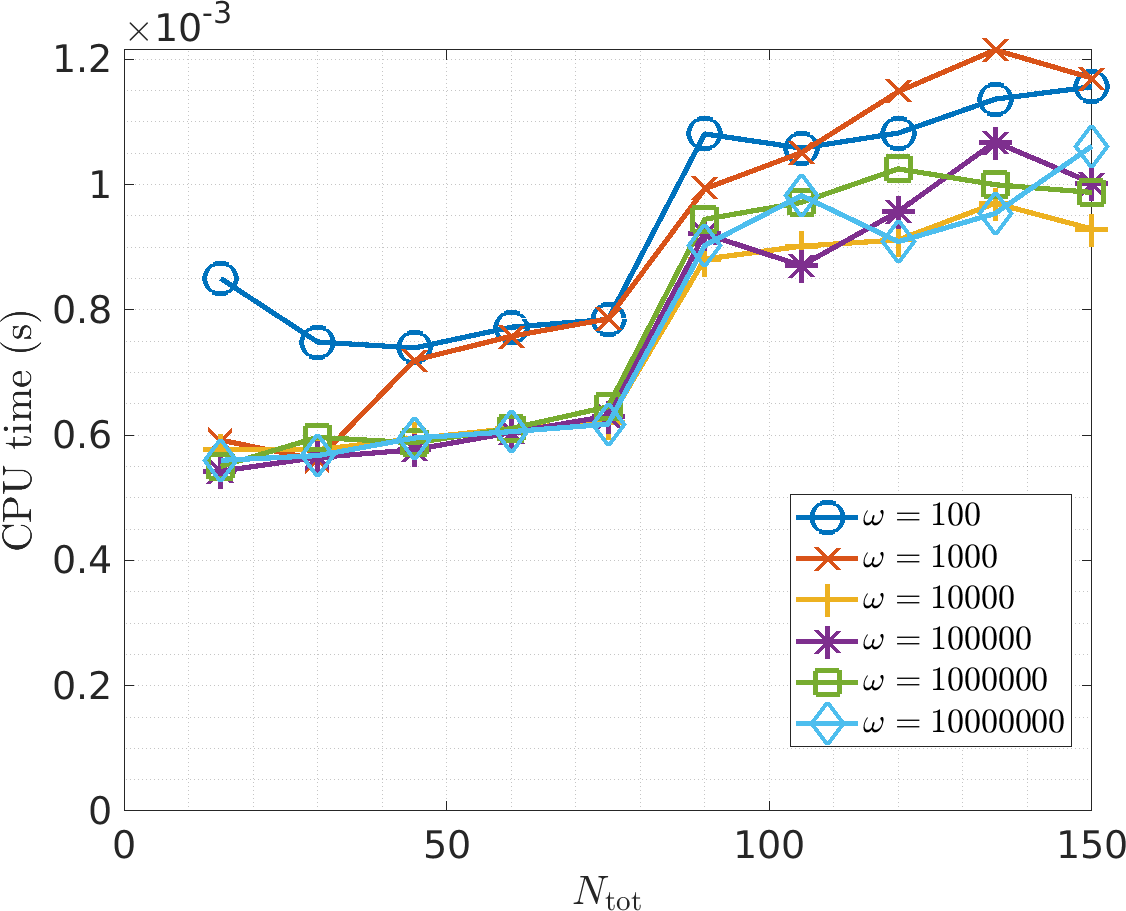}}
	\caption{Accuracy (a) and timings (b) for \eqref{eq:Coalescence} for $p=6$ and $r=0.01$. }
	\label{fig:coalescencecpur100p6a0}
\end{figure}

\begin{figure}[t!]
	\centering
		\subfloat[$N=10$]{
	\includegraphics[width=0.5\linewidth]{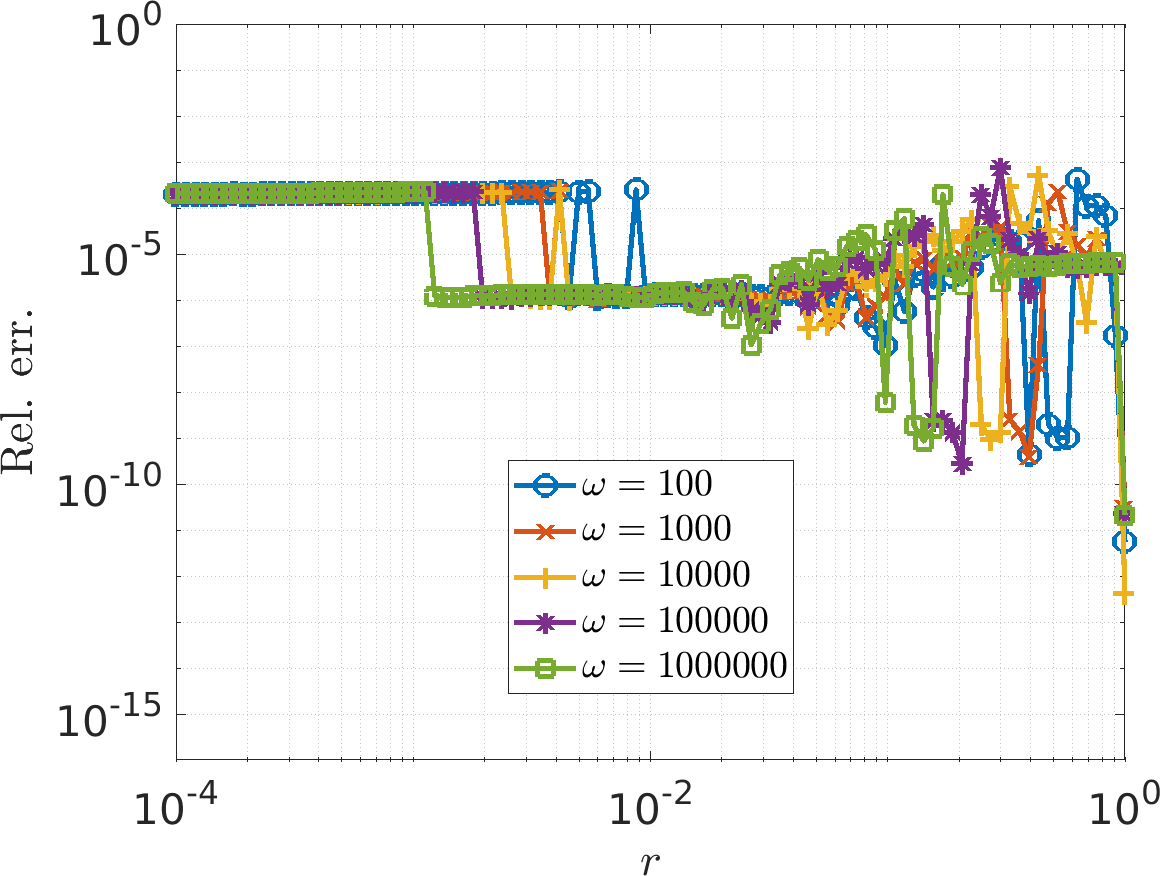}}
	\subfloat[$N=50$]{
	\includegraphics[width=0.5\linewidth]{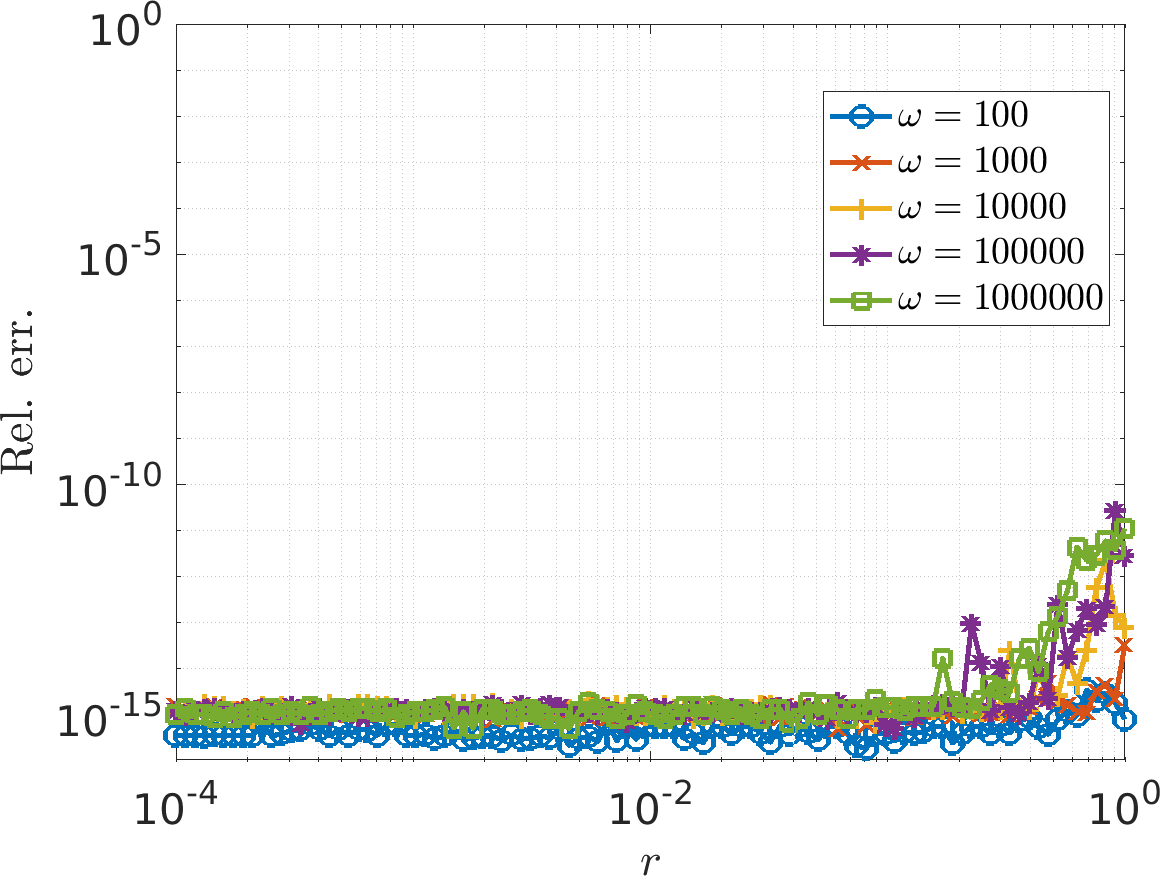}}
	\caption{Accuracy for \eqref{eq:Coalescence} as $r\to0$ for $p=6$, for $N=10$ (a) and $N=50$ (b).}
	\label{fig:coalescencetozero10p6}
\end{figure}

In Figure \ref{fig:coalescencecpur100p6a0} we show convergence (with respect to a PathFinder reference with $N=500$) and CPU times (averaged over 100 runs) for fixed $r=0.01$. We see that both the error and the CPU time are essentially independent of $\omega$ in this case. 
In Figure \ref{fig:coalescencetozero10p6} we plot errors for two fixed $N$ values $N=10,50$, as a function of $r$. We observe that as $r\to0$, the error stays bounded. 
For $N=10$ the error jumps up between $r=10^{-3}$ and $r=10^{-2}$, at a point depending on $\omega$. This represents the point at which the balls around the stationary points merge, resulting in a reduction of $N_{\rm tot}$, and hence a reduction in accuracy. But after this point we observe no further reduction in accuracy as $r\to0$. We remark that for sufficiently small $r>0$ the six stationary points are numerically indistinguishable, but this isn't a problem for our algorithm because in that case the problem will be treated identically to that of a monomial phase. 

\subsection{Canonical cuspoid integrals and their generalisations}
\label{sec:Applications}

In this section we show how our algorithm can be applied to the computation of some of the canonical integrals catalogued by Berry and Howls in \cite[\S36]{DLMF}, which, as mentioned already in \S\ref{sec:intro}, are 
of fundamental importance in numerous application areas including optics, acoustics and quantum mechanics. 

In this context, our algorithm is related to that of \cite{KiCoHo:00}, where an adaptive contour deformation approach was applied to evaluate the cuspoid integrals considered in \S\ref{sec:Cuspoid}. The algorithm in \cite{KiCoHo:00} is similar in spirit to our approach, in that it deforms the integration contour so that it terminates in valleys at infinity, and splits the contour into portions close to stationary points plus portions away from stationary points. However, in contrast to our approach, the algorithm in \cite{KiCoHo:00} does not attempt to trace SD contours, and hence is susceptible to rounding errors associated with the ``violent'' behaviour of the exponential factor $\ee^{\ii \omega g(z)}$ when one is not on a true SD contour - see \cite[\S2]{KiCoHo:00}. 
Furthermore, while the algorithm in \cite{KiCoHo:00} was specialised to the case of integration over the real line, our algorithm can handle much more general contours, as we illustrate in \S\ref{sec:Generalisations}.

\subsubsection{Cuspoid integrals}
\label{sec:Cuspoid} 
 
The so-called ``cuspoid integrals'' listed in \cite[\S36.2.4]{DLMF} are all of the form \eqref{eq:I} with polynomial phase $g$ and unit amplitude $f\equiv 1$, unit frequency $\omega=1$, and integration along the real line. Our algorithm is ideally suited to the evaluation of these integrals, and to demonstrate this we compute two of them. In the notation of \cite[\S36]{DLMF}, we consider the cusp catastrophe integral
\begin{align}
\label{eq:Pearcey}
\Psi_2(x,y) = P(y,x) = \int_{-\infty}^\infty \ee^{\ii (t^4+yt^2 + x t)}\,\dd t,
\end{align}
where $P$ is the Pearcey function, 
and the swallowtail catastrophe integral
\begin{align}
\label{eq:Swallowtail}
\Psi_3(x,y,z) = \int_{-\infty}^\infty \ee^{\ii (t^5+zt^3 + yt^2 + x t)}\,\dd t.
\end{align}
Both exhibit coalescence of stationary points on certain algebraic varieties (see \cite[\S36.5(ii)]{DLMF}) on which both the first and second derivatives of the phase function vanish. 
In the case of \eqref{eq:Pearcey} this occurs when
\begin{align}
\label{eq:PearceyCoalescence}
y = -\frac{3}{2}|x|^{2/3},
\end{align}
and for \eqref{eq:Swallowtail} this occurs when
\begin{align}
\label{eq:SwallowtailCoalescence}
400x^3-360x^2z^2-135y^4-27y^2z^3 + 540xy^2z + 81xz^4 = 0.
\end{align}

\begin{figure}[t!]
\subfloat[$|\Psi_2(x,y)|$]{\includegraphics[width=.5\linewidth]{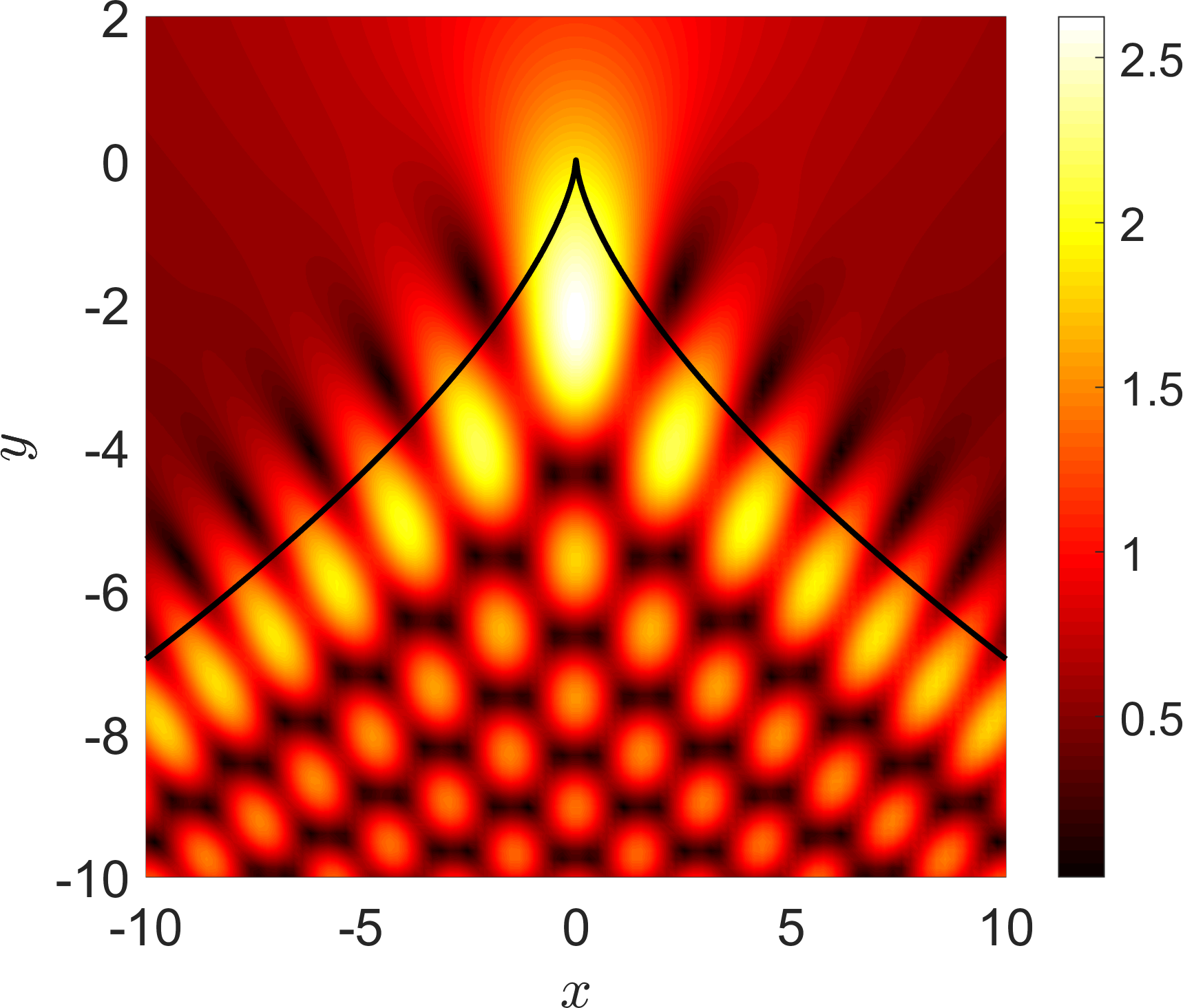}}
\hspace{2mm}\subfloat[$|\Psi_3(x,y,-7.5)|$]{\includegraphics[width=.5\linewidth]{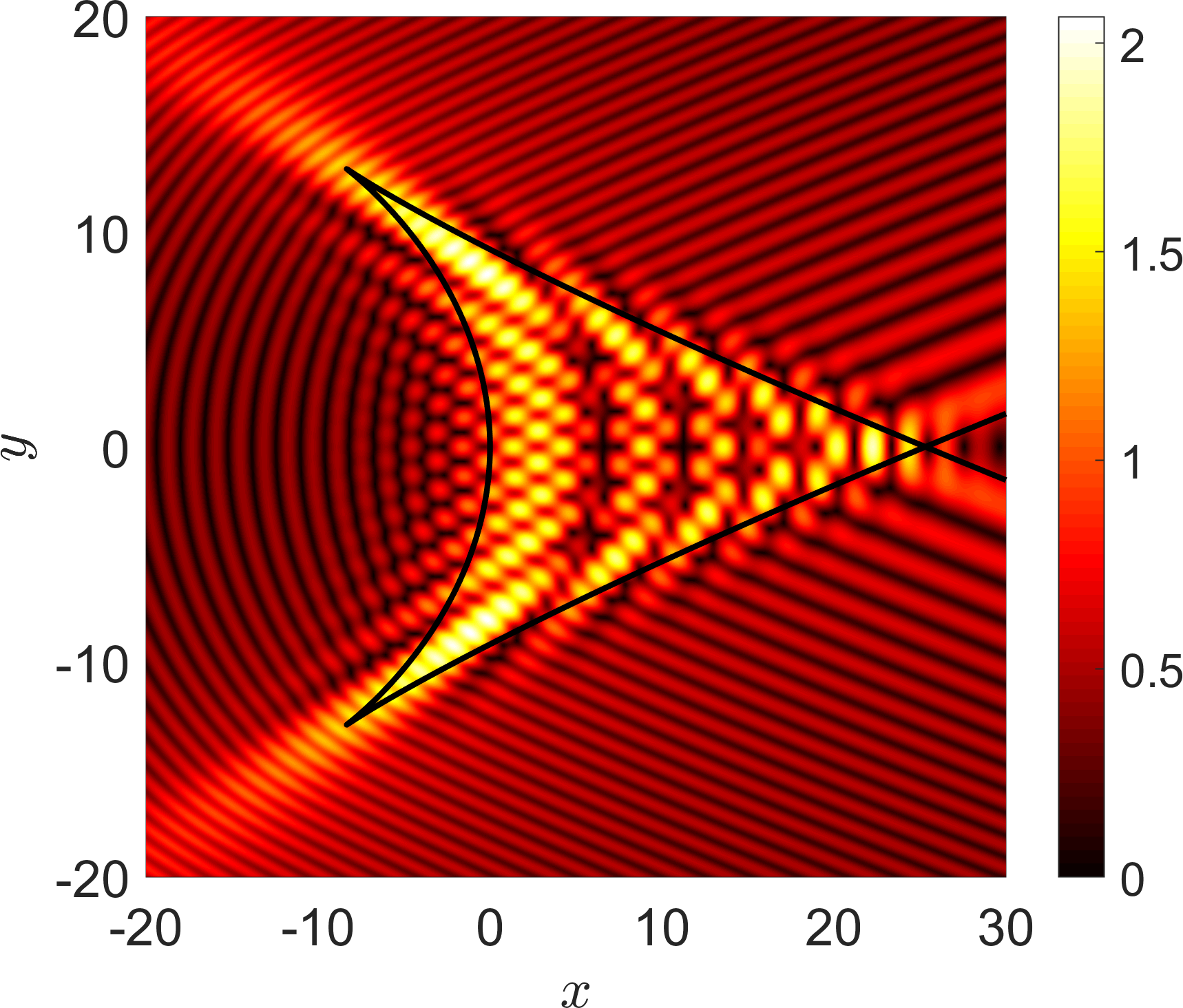}}
\caption{Magnitude plots of \eqref{eq:Pearcey} and \eqref{eq:Swallowtail}, with coalescence curves \eqref{eq:PearceyCoalescence} and \eqref{eq:SwallowtailCoalescence} (the latter with $z=-7.5$) superimposed in black. 
The computational grid was of size $100\times100$ for (a) and $500\times500$ for (b).}
\label{fig:DLMF}
\end{figure}

The integrals \eqref{eq:Pearcey} and \eqref{eq:Swallowtail} can be computed in Pathfinder via the commands
\begin{align*}
&\texttt{Psi2 = @(x,y) PathFinder(pi,0,[],[1 0 y x 0],...}\\
& \qquad \qquad\qquad\qquad \qquad \qquad\qquad \texttt{1,N,\textquotesingle infcontour\textquotesingle,[true true])}\\
&\texttt{Psi3 = @(x,y,z) PathFinder(pi,0,[],[1 0 z y x 0],...}\\
& \qquad \qquad\qquad\qquad \qquad \qquad\qquad \texttt{1,N,\textquotesingle infcontour\textquotesingle,[true true])}
\end{align*}
Figure \ref{fig:DLMF} shows plots of the magnitude of \eqref{eq:Pearcey} and \eqref{eq:Swallowtail} (the latter over the plane $z=-7.5$), computed using PathFinder with the default settings and $N=50$. The plots agree qualitatively with those presented in \cite[Figs 36.3.1 \& 36.6.5]{DLMF}, and, for \eqref{eq:Pearcey}, agree quantitatively (to all five decimal places presented) with the values presented in \cite[Table 1]{KiCoHo:00}. 
Computation times on a small desktop computer (Intel i7-4790, 32GB RAM) were less than a minute for the cusp (which required the computation of 10000 instances of \eqref{eq:Pearcey}, averaging $0.005$s per instance) and less than an hour for the swallowtail (which required 250000 instances of \eqref{eq:Swallowtail}, averaging $0.01$s per instance). 

\subsubsection{Generalisations}
\label{sec:Generalisations}

In \cite{HeOcSm:19} the authors considered a family of generalisations of certain canonical cuspoid integrals, with integration no longer over the real line, but rather over a complex contour starting and ending at valleys at infinity, and possibly with a non-unit amplitude function. 

A specific aim of \cite{HeOcSm:19} was to investigate the relevance of such integrals to the study of the so-called ``inflection point problem'', a canonical problem in wave scattering originally introduced over 50 years ago by Popov in \cite{Popov79}. 
This problem, which remains unsolved in closed form, concerns two-dimensional time-harmonic wave propagation near a boundary with an inflection point, and seeks a solution for the wave field near the inflection point that describes the transition from an incoming ``whispering gallery wave'' supported on the concave portion of the boundary, to outgoing ``creeping waves'' along the convex portion of the boundary, along with a scattered ``searchlight'' beam (for details and further references see \cite{smyshlyaev2022searchlight}).

In this context, in \cite[\S3.3]{HeOcSm:19} the authors studied the family of integrals 
\begin{align}
\label{eq:AijDef}
A_{ij}(x,y) = 
\int_{\Gamma_{ij}} f(t) \ee^{\ii(2t^5/5 - xt^4/2 -yt^2)}\,\dd t,
\end{align}
where $f(t)$ is some amplitude to be specified, and $\Gamma_{ij}$
denotes any contour from valley $v_i$ to valley $v_j$, where $v_j:=(2(j-1)+1/2)\pi/5$, $j=1,\ldots,5$. 
These integrals have stationary point coalescence on the cubic curve $y+4x^3/27=0$, which suggests that, by appropriately choosing $f$ and $\Gamma_{ij}$, they might exhibit certain features of the solution of the inflection point problem. Indeed, in \cite[\S4]{HeOcSm:19} it was shown that as $x\to-\infty$ near the cubic curve, the integral $A_{32}$ has the character of an incoming whispering gallery type wave, and that, as $x\to+\infty$ near the cubic curve, the integral $A_{52}$ has the character of an outgoing creeping wave. However, plots of the resulting fields could not be presented in \cite{HeOcSm:19} due to the lack of a suitable numerical evaluation method and implementation. 

\begin{figure}[t!]
\subfloat[$|A_{32}(x,y)|$]{\includegraphics[width=0.45\linewidth]{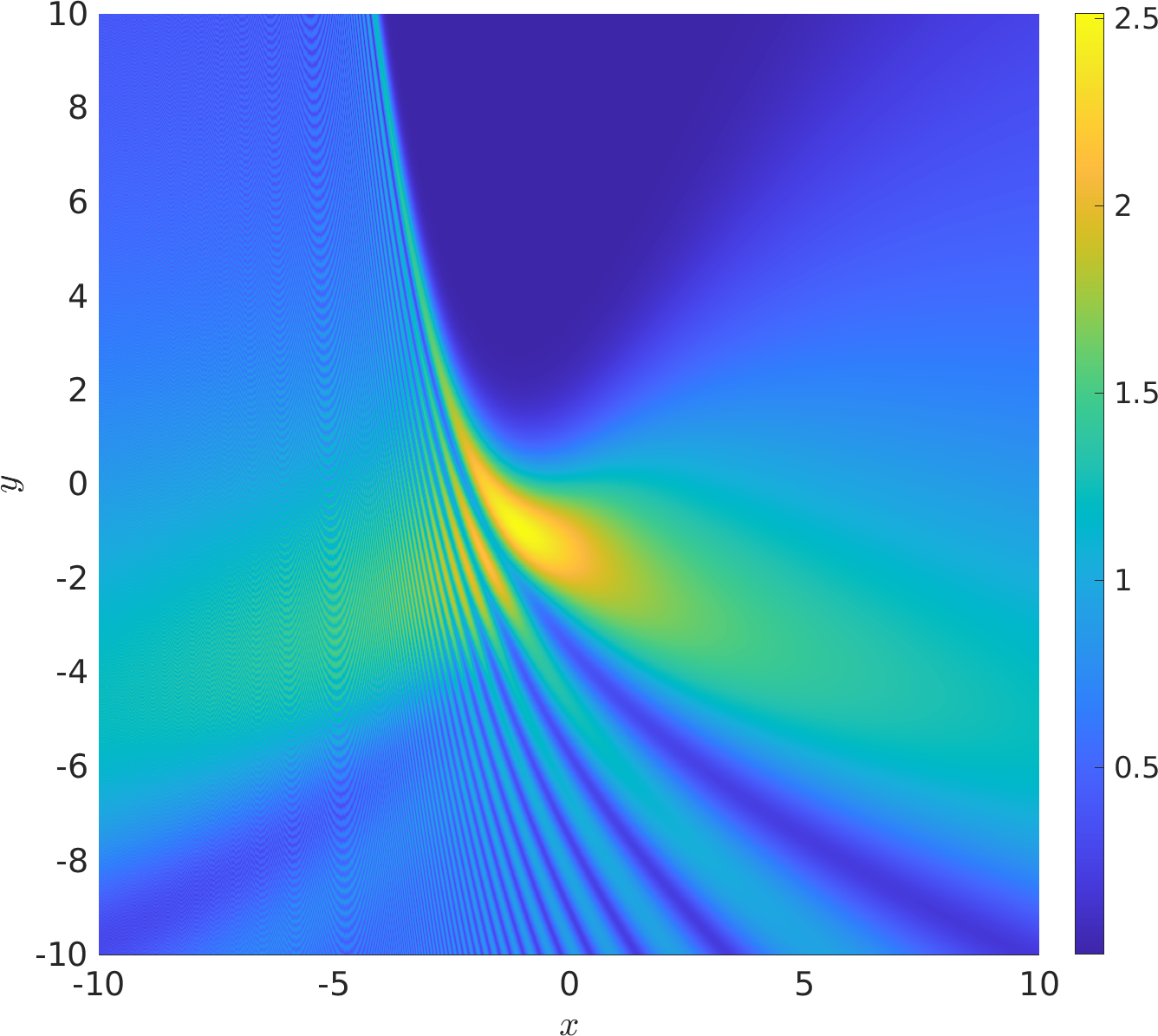}}
\hspace{2mm}
\subfloat[$|A_{52}(x,y)|$]{\includegraphics[width=0.45\linewidth]{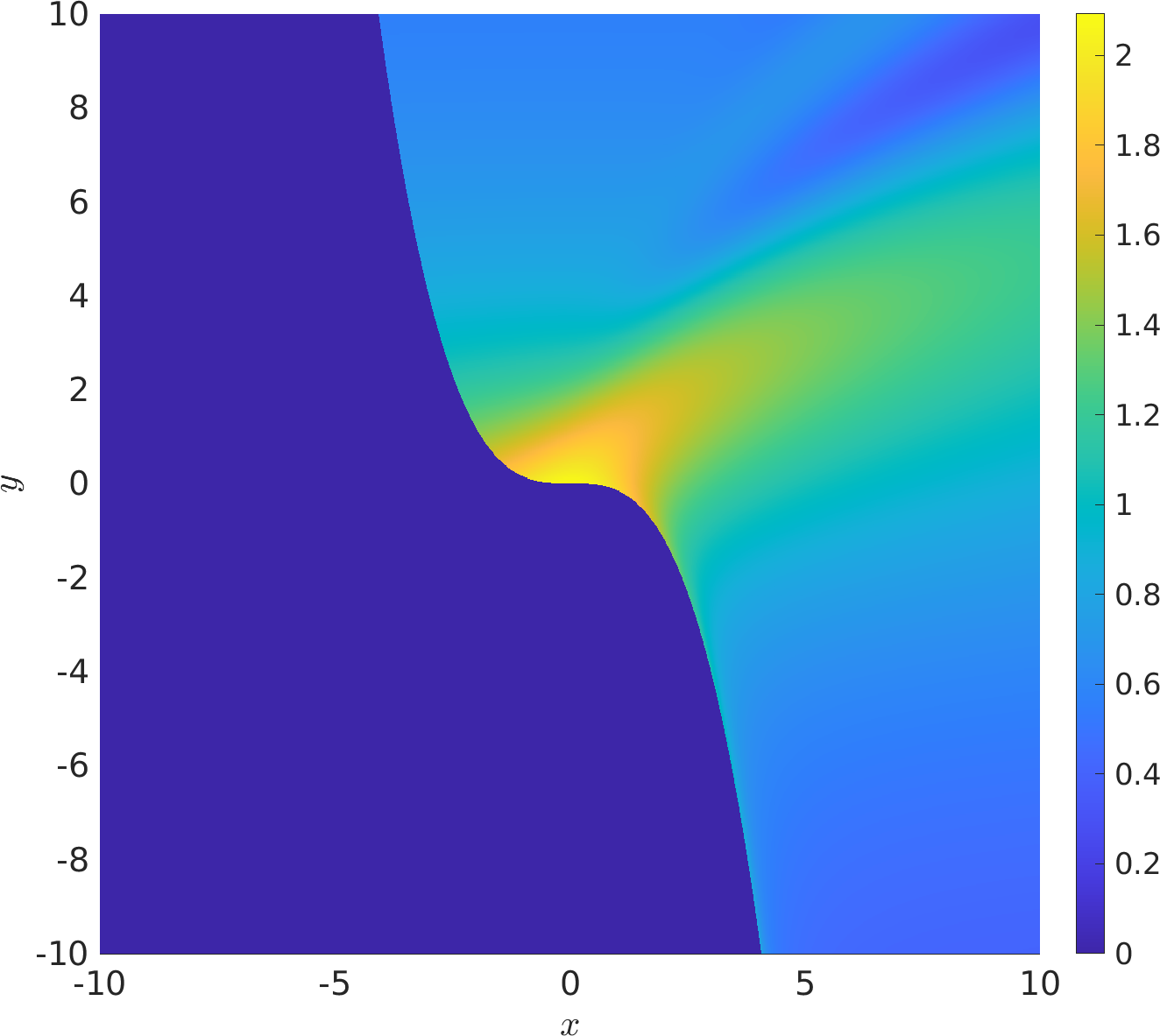}}\\
\subfloat[{$\Re[A_{32}(k^{1/5}x_0,k^{3/5}y_0)\ee^{\ii k x_0}]$}]{\includegraphics[width=\linewidth]{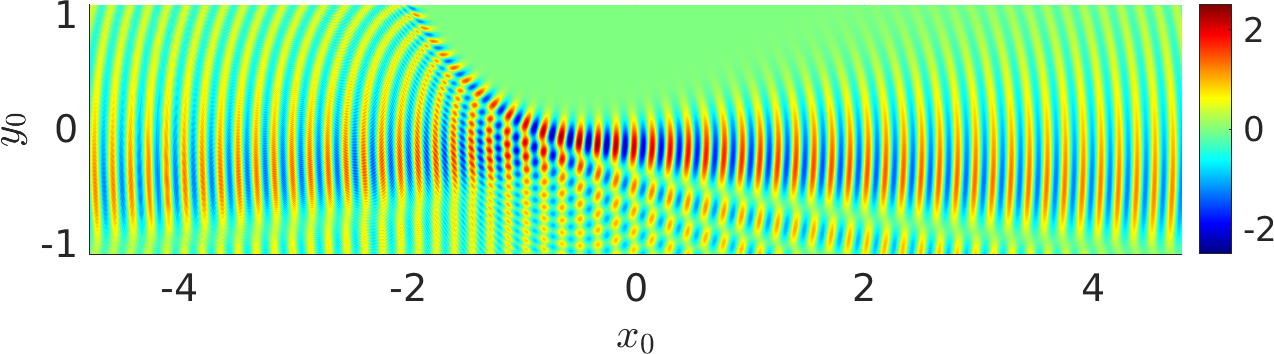}}\\
\subfloat[{$\Re[A_{52}(k^{1/5}x_0,k^{3/5}y_0)\ee^{\ii k x_0}]$}]{\includegraphics[width=\linewidth]{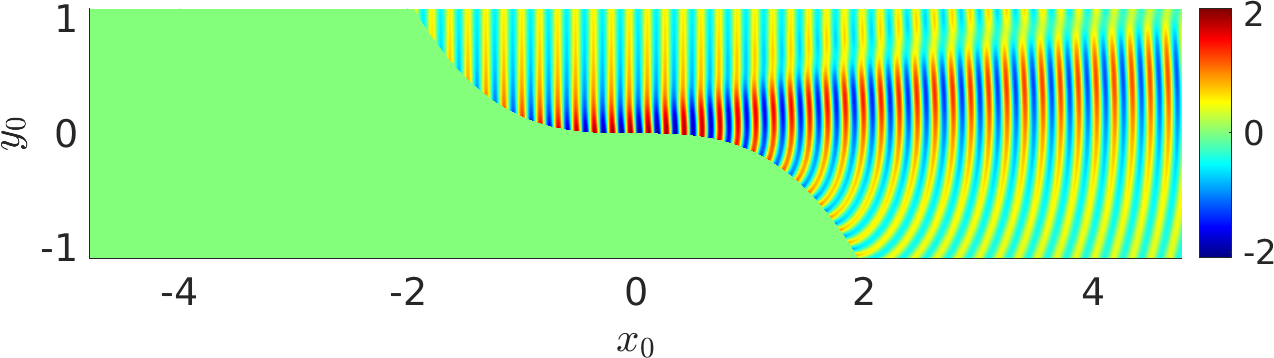}}
\caption{Plots of \eqref{eq:AijDef} with $f\equiv 1$, along with the associated approximate Helmholtz equation solutions for $k=40$.}
\label{fig:PWE}
\end{figure}

Using PathFinder we are able to remedy this. In Figures \ref{fig:PWE}(a) and \ref{fig:PWE}(b) we provide plots of the magnitude of $A_{32}$ and $A_{52}$ with $f\equiv 1$. To evaluate the integrals we used the PathFinder code
\begin{align*}
&\texttt{A32 = @(x,y) PathFinder(9*pi/10,pi/2,[],[2/5 -x/2 0 -y 0 0],...}\\
& \qquad \qquad\qquad\qquad \qquad \qquad\qquad \texttt{1,N,\textquotesingle infcontour\textquotesingle,[true true])}\\
&\texttt{A52 = @(x,y) PathFinder(17*pi/10,pi/2,[],[2/5 -x/2 0 -y 0 0],...}\\
& \qquad \qquad\qquad\qquad \qquad \qquad\qquad \texttt{1,N,\textquotesingle infcontour\textquotesingle,[true true])}
\end{align*}
We only plot $A_{52}$ above the cubic curve $y+4x^3/27=0$, because below this curve $A_{52}$ becomes exponentially large (cf.\ \cite[Fig.~12(i)]{HeOcSm:19}). In Figures \ref{fig:PWE}(c) and \ref{fig:PWE}(d) we present corresponding plots of the modulated plane wave
\[u(x_0,y_0) = A_{ij}(x,y) \ee^{\ii k x_0},\] 
where $(x_0,y_0)$ are outer variables, related to the inner variables $(x,y)$ by $x=k^{1/5}x_0$, $y=k^{3/5}y_0$, which is an asymptotic solution of the Helmholtz equation $\Delta u + k^2 u = 0$ as $k\to\infty$ in the region $x_0=O(k^{-1/5})$, $y_0=O(k^{-3/5})$ \cite[\S1]{HeOcSm:19}. 
Here one observes the predicted incoming whispering gallery type behaviour of $A_{32}$ near the top of Figure \ref{fig:PWE}(c) between $x_0=-2$ and $x_0=-1$, with oscillations giving way to an exponentially small field in the caustic shadow, and the predicted creeping wave type behaviour of $A_{52}$ near the bottom of Figure \ref{fig:PWE}(d) between $x_0=1$ and $x_0=2$, with waves propagating along the cubic curve, shedding rays tangentially.

In ongoing and future studies we plan to use PathFinder to further investigate the properties of integrals of the form \eqref{eq:AijDef}, and generalisations involving different choices of $f$ and higher degree phase functions (see \cite{HeOcSm:19}), 
which we hope may shed new light on the inflection point problem and related problems in high frequency wave propagation. 

\section*{Acknowledgements}
The authors acknowledge support from KU Leuven project C14/15/055 (AG) and EPSRC projects EP/S01375X/1 and EP/V053868/1 (DPH and AG). We are grateful to Alex Townsend, Nick Trefethyn, Marcus Webb, Sheehan Olver and Samuel Groth for helpful discussions in relation to this project.

\bibliographystyle{siam}
\bibliography{refs}

\end{document}